\documentclass[12pt,twoside,reqno]{amsart}
\usepackage{amsmath}
\usepackage{amsfonts}
\usepackage{amssymb}
\usepackage{color}
\usepackage{mathrsfs}
\usepackage{multicol}
\usepackage{graphicx}
\usepackage{cite}
\usepackage{geometry}
\usepackage{marginnote}
\usepackage{todonotes}
\usepackage{tikz}
\usepackage{pgfplots}
\pgfplotsset{compat=1.10}
\usepgfplotslibrary{fillbetween}
\usetikzlibrary{patterns}
\DeclareMathOperator*{\argmin}{argmin}
\newtheorem{theorem}{Theorem}[section]
\newtheorem{corollary}[theorem]{Corollary}
\newtheorem{lemma}[theorem]{Lemma}
\newtheorem{proposition}[theorem]{Proposition}
\newtheorem{definition}[theorem]{Definition}
\newtheorem{remark}[theorem]{Remark}
\numberwithin{equation}{section}
\allowdisplaybreaks[1]
\newcommand{\eqntag}{\addtocounter{equation}{1}\tag{\theequation}} 
\begin{document}
\title{Global existence and boundedness for a degenerate chemotaxis system with indirect signal production via minimizing movement schemes} 
\thanks{\texttt{ORCID: 0000-0002-1027-9485(TH), 0000-0003-3091-8085 (PhL)}}
\author{Tatsuya Hosono}
\address{Osaka Central Advanced Mathematical Institute, Osaka Metropolitan University,
	Osaka 558-8585, Japan 
	\& Institut de Math\'ematiques de Toulouse (IMT) UMR5219, 
	Universit\'e Toulouse III - Paul Sabatier, F-31062 Toulouse Cedex 9, France}
\email{tatsuya.hosono@omu.ac.jp}
	
\author{Philippe Lauren\c{c}ot}
\address{Laboratoire de Math\'ematiques (LAMA) UMR 5217, Universit\'e Savoie Mont Blanc, CNRS, F-73000 Chamb\'ery, France}
\email{philippe.laurencot@univ-smb.fr}
	

\date{\today}

\begin{abstract}
Global existence and boundedness of weak solutions for a fully parabolic degenerate chemotaxis system with indirect signal production are proved for any initial data in the subcritical case and under smallness conditions in the critical and supercritical cases. To construct weak solutions, a time discrete scheme is set up, for which the first equation has a gradient flow structure with respect to the $2$-Wasserstein distance, while the other two equations feature an $L^2$-variational structure. The proof relies in particular on the flow interchange method and discrete maximal regularity.
\end{abstract}
%
\maketitle
%
%
%
\pagestyle{myheadings}
\markboth{\sc{T. Hosono \& Ph. Lauren\c cot}}{\sc{Degenerate chemotaxis system with indirect signal production}}

\section{Introduction}\label{sec.1}

Let $m>1$. Global existence of bounded weak solutions to the degenerate fully parabolic chemotaxis system with indirect signal production
\begin{equation}
\left\{
\begin{aligned}
	&\partial_t u = \mathrm{div}\left( \nabla u^m - u \nabla v \right),
	& t>0,\, &x\in \mathbb{R}^d,
	\\
	&\partial_t v = \Delta v - v +w,
	&t>0,\, &x\in \mathbb{R}^d,
	\\
	&\partial_t w = \Delta w - w +u,
	&t>0,\, &x\in \mathbb{R}^d,
	\\
	&(u,v,w)(0,x)=(u_{0},v_0,w_0)(x),
	&\, &x\in \mathbb{R}^d,
\end{aligned}
\right.
\label{eqn;DCI0}
\end{equation}
is established for suitable values of $m>1$ and $M=\|u_0\|_1>0$ depending on the space dimension~$d\ge 1$. Chemotaxis describes the directed movement of cells or microorganisms with density~$u$ in response to a chemical signal with concentration $v$ and has been extensively studied since the seminal work of Keller \& Segel \cite{KeSe1971}. While the classical Keller-Segel model assumes that the chemoattractant is produced directly by the cells, several variants have been proposed to account for more complex biological signaling pathways. Among them, chemotaxis systems with indirect signal production introduce an intermediate substance~$w$ through which the chemoattractant is generated, leading to the fully parabolic system~\eqref{eqn;DCI0}. As for the degenerate fully parabolic Keller-Segel system, dimensional analysis reveals that there is a critical value $m^*:=2(d-2)/d$ for~\eqref{eqn;DCI0}, which separates different dynamical behaviors: for subcritical values of $m$ ($m>m^*$), global existence is expected while finite time blowup is likely to occur for supercritical values of $m$ ($m<m^*$). Besides, an additional threshold phenomenon is expected in the critical case $m=m^*$, where it is the mass $M=\|u_0\|_1$ which governs the dynamics. Large values of $M$ lead to finite time blowup, while a sufficiently small mass guarantees global existence. Observe that $m^*>1$ if and only if $d\ge 5$ and linear diffusion $m=m^*=1$ is critical in space dimension $d=4$. This fact is confirmed in the seminal work by Fujie \& Senba \cite{FuSe2017}, where global existence of bounded solutions to~\eqref{eqn;DCI0} with $m=1$ (but in a bounded domain $\Omega$ with suitable boundary conditions) is shown when $d=4$ and $M\in \big(0,(8\pi)^2\big)$, as well as when $d\in\{1,2,3\}$ ($m^*<1$ when $d\le 3$). Moreover, the occurrence of finite or infinite time blowup is established for suitable initial data in \cite{FuSe2019} when $d=4$ and $M>(8\pi)^2$, while finite time blowup for suitable initial data, but without any lower bound on $M$, is obtained in \cite{MLL2025} for $d\ge 5$.  When $m=1$ and $d=4$, the Cauchy problem~\eqref{eqn;DCI0} is studied in \cite{HoLa2025} and the existence of a unique classical solution is shown for $M\in \big(0,(8\pi)^2\big)$, along with the boundedness of these solutions under the additional (non-optimal) constraint $M<(8\pi)^2/\sqrt{3}$. We also refer to \cite{Hos2026} for global existence and boundedness of classical solutions to the Cauchy problem~\eqref{eqn;DCI0} in low space dimensions $d\in\{1,2,3\}$. Less attention has been paid to the quasilinear case $m>1$: for the initial-boundary value problem~\eqref{eqn;DCI0} in a bounded domain, supplemented with homogeneous Neumann boundary conditions in the non-degenerate case where $u^m$ is replaced by $(1+u)^m$, global existence of bounded classical solutions is proved in \cite[Theorem~1.1]{DiWa2019} when $m>m^*$ in space dimension $d\ge 2$ and $m>-1$ when $d=1$. More recently, global existence of weak solutions to the Cauchy problem~\eqref{eqn;DCI0} has been established by Mimura \cite{Mim2024b} in the subcritical and critical cases $m\ge m^*$, assuming further that
\begin{equation}
	\;\text{ either }\; m\ge 2, \;\text{ or }\; m\in (1,2) \;\text{ and }\; m \ge \frac{3d}{2(d+1)}. \label{hyp_mimura}
\end{equation}
As pointed out in \cite{Mim2024b}, we have $3d/[2(d+1)]>m^*$ for $d\le 5$, so that the critical exponent $m^*$ does not satisfy~\eqref{hyp_mimura} when $d=5$, as well as an open interval of subcritical values when $1\le d\le 5$.

One aim of this work is to complete the analysis performed in \cite{Mim2024b} by showing that there is a global weak solution to~\eqref{eqn;DCI0} for all critical and subcritical values of $m>1$. Our results actually go far beyond this extension and provide improvements in two directions:  on the one hand, we show uniform boundedness of solutions in the critical and subcritical cases and derive additional regularity estimates on the solutions, including an $L^\infty$-estimate on the first component $u$. On the other hand, we also investigate the existence of global and bounded weak solutions in the supercritical case $m<m^*$, provided that $m>1$ lies above the Sobolev exponent $m_*:=2d/(d+4)$. 

Before stating our results, we rescale~\eqref{eqn;DCI0} in order to transform the initial condition $u_0$ to a probability measure. Then, the original mass $M$ appears as a parameter in the resulting system, which reads
\begin{equation}
	\left\{
	\begin{aligned}
		&\partial_t u = \mathrm{div}\left( M^{m-1} \nabla u^m - M u \nabla v \right),
		& t>0,\, &x\in \mathbb{R}^d,
		\\
		&\partial_t v = \Delta v - v +w,
		&t>0,\, &x\in \mathbb{R}^d,
		\\
		&\partial_t w = \Delta w - w +u,
		&t>0,\, &x\in \mathbb{R}^d,
		\\
		&(u,v,w)(0,x)=(u_{0},v_0,w_0)(x),
		&\, &x\in \mathbb{R}^d,
	\end{aligned}
	\right.
	\label{eqn;DCI}
\end{equation}
with non-negative initial conditions 
\begin{equation}
	(u_0,v_0,w_0)\in \mathcal{P}_2(\mathbb{R}^d)\times W_+^{2,2}(\mathbb{R}^d)\times L_+^2(\mathbb{R}^d), \label{hyp_CI}
\end{equation}
where $\mathcal{P}_2(\mathbb{R}^d)$ denotes the space of probability measures on $\mathbb{R}^d$ with finite second moment. Next, as~\eqref{eqn;DCI} features a degenerate diffusion, classical solutions are not likely to exist globally and we state below the definition of a weak solution to~\eqref{eqn;DCI} to be used in the sequel.

\begin{definition}\label{def.ws}
Let $m>1$, $M>0$ and consider initial data $(u_0,v_0,w_0)$ satisfying~\eqref{hyp_CI}, as well as
\begin{equation}
	u_0\in L^\infty(\mathbb{R}^d), \quad v_0\in W^{2,m}(\mathbb{R}^d)\cap W^{2,\infty}(\mathbb{R}^d), \quad w_0\in W^{2,m}(\mathbb{R}^d)\cap W^{2,\infty}(\mathbb{R}^d). \label{ph01}
\end{equation}
A global weak solution to~\eqref{eqn;DCI} is a triple $(u,v,w)$ of non-negative functions satisfying
\begin{equation*}
	u\in C\big([0,\infty),\mathcal{P}_{2}(\mathbb{R}^d)\big) , \quad v\in C\big([0,\infty),W_+^{1,2}(\mathbb{R}^d)\big), \quad w\in C\big([0,\infty),L_+^{2}(\mathbb{R}^d)\big),
\end{equation*}
along with
\begin{align*}
	& u\in L^2\big((0,T)\times\mathbb{R}^d\big), \quad \nabla u^m \in L^2\big((0,T)\times\mathbb{R}^d;\mathbb{R}^d\big), \\
	& v \in L^2\big((0,T),W^{2,2}(\mathbb{R}^d)\big), \quad w \in L^2\big((0,T),W^{2,2}(\mathbb{R}^d)\big),
\end{align*}
and, for any $\varphi\in C_c^\infty\big((0,\infty)\times\mathbb{R}^d\big)$, 
\begin{align*}
	\int_0^\infty \int_{\mathbb{R}^d} u \partial_t\varphi\ dxdt & = \int_0^\infty \int_{\mathbb{R}^d} \left[ \nabla u^m - u\nabla v \right] \cdot \nabla\varphi \ dxdt, \\
	\partial_t v & = \Delta v - v +w \;\text{ a.e. in }\; (0,\infty)\times\mathbb{R}^d,\\
	\partial_t w & = \Delta w - w +u \;\text{ a.e. in }\; (0,\infty)\times\mathbb{R}^d.
\end{align*}
\end{definition}

Let us point here that Definition~\ref{def.ws} requires slightly more regularity on $\nabla u^m$ than the definition of a weak solution used in \cite{Mim2024b} but this additional property of $\nabla u^m$ is actually derived in the course of the existence proof. The main result we present now is the existence of a global bounded weak solution to~\eqref{eqn;DCI} under appropriate conditions on $m>1$ and $M>0$, according to the range of the space dimension $d\ge 1$.

\begin{theorem}\label{thm.gbe}
Let $m>1$, $M>0$ and consider initial data $(u_0,v_0,w_0)$ satisfying~\eqref{hyp_CI} and~\eqref{ph01}. Then~\eqref{eqn;DCI} has a global weak solution which satisfies additionally 
\begin{equation}
	u\in L^\infty\big((0,\infty)\times\mathbb{R}^d)\big), \quad v\in L^\infty\big((0,\infty),W^{2,\infty}(\mathbb{R}^d)\big), \quad w\in L^\infty\big((0,\infty),W^{1,\infty}(\mathbb{R}^d)\big) \label{Linfty}
\end{equation} 
in the following cases:
\begin{itemize}
	\item [\textbf{(m1)}] $d\ge 5$ and $m>m^* = (2d-4)/d$ or $1\le d\le 4$ (subcritical case);
	\item [\textbf{(m2)}] $d\ge 5$, $m=m^*$ and $M\in \big(0,M^*\big)$, where the threshold mass $M^*$ is related to the optimal constant in a functional inequality and is given by $M^*=(m^*-1)K_{m^*}^2/2$ with
	\begin{equation*}
		K_{m^*}^2 = \sup_{u\in (L_+^1\cap L^{m^*})(\mathbb{R}^d), \ \|u\|_1=1} \frac{\displaystyle{\int_{\mathbb{R}^d} u (-\Delta)^{-2}u \ dx}}{\|u\|_{m^*}^{m^*}},
	\end{equation*} 
	see~\eqref{ph16}, \eqref{ph17} and~\eqref{ph19ba} (critical case); 
	\item [\textbf{(m3)}] $d\ge 5$, $m\in (m_*,m^*)$ and $\|u_0\|_m^m < z_m$, where $z_m>0$ depends only on $d$, $m$ and $M$ and is defined by~\eqref{ph19c}. In that case, $\|u(t)\|_m\le \|u_0\|_m$ for $t\ge 0$ (supercritical case).
\end{itemize}
\end{theorem}

As already mentioned, existence of a global weak solution to~\eqref{eqn;DCI} in the subcritical case~\textbf{(m1)} and the critical case~\textbf{(m2)} is shown in \cite{Mim2024b} but only for $d\ge 6$. Theorem~\ref{thm.gbe} thus extends the existence result to the full subcritical and critical ranges as described in~\textbf{(m1)} and~\textbf{(m2)} and also provides the boundedness~\eqref{Linfty} of the solution. Besides, we obtain in~\textbf{(m3)} the existence of a global bounded weak solution in the supercritical case $m\in (m_*,m^*)$ under a smallness assumption on $\|u_0\|_m$. The outcome of~\textbf{(m3)} is perfectly consistent with (and actually inspired by) similar results already available for the degenerate Keller-Segel system \cite{CLW2012, ChWa2014, KNO2015, LiWa2026, Mim2024a, Oga2011}.

The starting point of the analysis is the availability of a Liapunov functional $\mathcal{L}$ for~\eqref{eqn;DCI}, which is given by
\begin{equation*}
	\mathcal{L}[u,v,w] := \frac{M^{m-2}}{m-1} \|u\|_m^m - \int_{\mathbb{R}^d} uv\ dx + \frac{\|v-\Delta v\|_2^2}{2} + \frac{\|w-v+\Delta v\|_2^2}{2}.
\end{equation*}
In view of the negative term in $\mathcal{L}$, which could be the dominating one, the functional $\mathcal{L}$ may be unbounded from below according to the values of $d$ and $m$, a feature which could lead to finite or infinite time singularities. The critical exponent $m^*$ is deeply connected with the boundedness from below of $\mathcal{L}$, as shown in \cite[Section~4]{Mim2024b}. Specifically, $\mathcal{L}$ is bounded from below for all $M>0$ in the subcritical range $m>\max\{1,m^*\}$, while this property only holds true for sufficiently small values of~$M$ for the critical exponent $m=m^*$, the threshold parameter $M^*$ being identified according to this criterion. A useful consequence of the boundedness from below of $\mathcal{L}$ is a control on the $L^m$-norm of $u$, as well as on the $W^{2,2}$- norm of $v$ and the $L^2$-norm of $w$. On the one hand, these estimates are the first step towards the derivation of further estimates, culminating eventually in $L^\infty$-estimates on~$(u,v,w)$. On the other hand, such a lower bound on $\mathcal{L}$ is not available in the supercritical range and the derivation of an estimate on the $L^m$-norm of $u$ requires a smallness condition on this norm. In fact, in view of $m<m^*<2$, one can control the quadratic negative term by $\|u\|_m^m$, provided the latter is sufficiently small. 

Let us now describe the proof of Theorem~\ref{thm.gbe}. While the classical and degenerate Keller-Segel systems have a gradient flow structure \cite{BCKKLL2015, BlLa2013, Mim2017}, the more intricate coupling arising from the indirect signal production mechanism seems to prevent the system~\eqref{eqn;DCI} from possessing such a structure, even though it has a Liapunov functional as already mentioned. However, as noticed in \cite{Mim2024b}, there is a suitable time discrete scheme of~\eqref{eqn;DCI} for which each equation has a variational structure, which turns out to be a very convenient approach to construct solutions, see also \cite{HTZ2025, Mim2017} for a similar method on other chemotaxis models. However, this time discrete scheme leads to a variational structure with respect to the $2$-Wasserstein distance on $\mathcal{P}_2(\mathbb{R}^d)$ for the first equation, which is combined with a classical $L^2$-variational structure for the other two equations, and it is well-known that handling the former requires some care. Nevertheless, we shall use the same time discrete scheme  here as well but, in contrast to \cite{Mim2024b},  we first establish regularity properties of the first component of the solutions to the discrete scheme by exploiting its minimizing properties, before deriving the corresponding Euler-Lagrange equation and showing the convergence. Proceeding in that way allows us in particular to employ the flow interchange technique developed in \cite{MMS2009}, along with discrete maximal regularity results \cite{APW2002, KLL2016}, to obtain additional integrability properties of the discrete approximation of $u$ and additional Sobolev regularity on the discrete approximations of $v$ and $w$. In turn, these improvements are instrumental in order to handle the full subcritical and critical ranges~\textbf{(m1)-(m2)} and to cope with the supercritical range~\textbf{(m3)}. In particular, they pave the way towards the boundedness of the discrete approximation of $u$, which is achieved here by a Moser's technique. The core of the paper is actually the derivation of several new estimates on the discrete approximations which are gathered in Theorem~\ref{thm.est} below. Let us also emphasize two salient features of our proof: on the one hand, the bounds obtained on the constructed weak solutions to~\eqref{eqn;DCI} are also available for the time discrete approximations and are valid uniformly with respect to the time step. On the other hand, the derivation of these bounds at the discrete level is based upon ``time'' weighted estimates, which are somewhat reminiscent of those used in \cite[Proposition~3.2]{IsYo2020} to prove the boundedness of weak solutions to the degenerate Keller-Segel system.

The next section is devoted to the time discrete scheme, which is introduced in~\eqref{ph02} and its properties. We begin with the derivation of the Liapunov functional at the discrete level, see Section~\ref{sec.2.1}, which is done along the lines of \cite[Proposition~3.1]{Mim2024b}. The derivation of a lower bound of the Liapunov functional involving the $L^m$-norm is next performed in Section~\ref{sec.2.2}, the proof depending heavily on the value of $m$, and the $L^m$-bound on $u$ is stated in Section~\ref{sec.2.3}. $L^\rho$-estimates on the discrete approximation of $u$ for $\rho\in (m,\infty)$ are the subject of Section~\ref{sec.2.4}, where the flow interchange method and discrete maximal regularity are used. Section~\ref{sec.2.5} is devoted to $L^\infty$-estimates, first on $w$, $\Delta v$, $\nabla w$ and $\nabla\Delta v$, and then on $u$, these estimates being collected in Theorem~\ref{thm.est}. Convergence of the discrete approximations is shown in Section~\ref{sec.3}, and Theorem~\ref{thm.gbe} as well, but we only provide a sketch of the proof and omit the details, since the arguments are very close to \cite[Sections~5-7]{Mim2024b}. We finally gather in the appendix several regularity properties of the solutions to the time discrete heat equation, including consequences of discrete maximal regularity, which are used in Section~\ref{sec.2}. 

\bigskip

For convenience, we recall below the definition of the critical values of $m$, which play an important role in the analysis.
\begin{equation*}
	m_*:= \frac{2d}{d+4} \;\;\text{ and}\;\; m^* := 2 - \frac{4}{d} \;\;\text{ for } d\ge 1. 
\end{equation*}
In particular, $m^*>m_*>1$ for $d\ge 5$, $m^*=m_*=1$ for $d=4$ and $m^*<m_*<1$ for $1\le d\le 3$.

\section{A discrete scheme}\label{sec.2}

The proof of Theorem~\ref{thm.gbe} is based on the minimizing movement scheme introduced in \cite{Mim2024b}. Although we employ the same scheme, our approach differs from that in \cite{Mim2024b}, as we first derive new discrete integrability estimates for the first component $u$ by combining the flow interchange technique with discrete maximal regularity, and then identify the corresponding Euler-Lagrange equations. These estimates are uniform with respect to the time step and provide the key ingredients for establishing both the global existence and the boundedness of weak solutions, which will be shown in the next section. We begin by introducing the admissible spaces in which the successive minimization problems are posed.

\begin{align*}
	X_0 & := \left\{ (u,v,w)\in (L_+^1\cap L^m)(\mathbb{R}^d) \times W_+^{2,2}(\mathbb{R}^d)\times L_+^2(\mathbb{R}^d)\ :\ \|u\|_1=1 \right\}, \\
	X_2 & := \left\{ (u,v,w)\in (\mathcal{P}_2\cap L^m)(\mathbb{R}^d) \times W_+^{2,2}(\mathbb{R}^d)\times L_+^2(\mathbb{R}^d) \right\} \subset X_0, \\
	Y_0 & := \left\{ (u,v)\in (L_+^1\cap L^m)(\mathbb{R}^d) \times \dot{W}_+^{2,2}(\mathbb{R}^d)\ :\ \|u\|_1=1 \right\}, 
\end{align*}
and
\begin{equation*}
	\mathcal{U}_m := \left\{
	\begin{array}{ll}
		(\mathcal{P}_2\cap L^m)(\mathbb{R}^d) & \;\;\text{ if }\;\; m\ge m^*, \\[1ex]
		\left\{ u\in (\mathcal{P}_2\cap L^m)(\mathbb{R}^d)\ :\ \|u\|_m \le z_m^{1/m} \right\} & \;\;\text{ if }\;\; m\in \big(m_*,m^*\big),
	\end{array} \right.
\end{equation*}
where $z_m>0$ is defined in~\eqref{ph19c} below. 

Let $\tau\in (0,1)$. Arguing as in \cite[Proposition~2.2]{Mim2024b}, the sequences $(u_n^\tau,v_n^\tau,w_n^\tau)_{n\ge 1}$ defined inductively by
\begin{subequations}\label{ph02}
\begin{align}
	w_n^\tau & = \argmin_{w\in W^{1,2}(\mathbb{R}^d)} \left\{ \frac{\big\|w-w_{n-1}^\tau\big\|_2^2}{2\tau} + \frac{\|\nabla w\|_2^2 + \|w\|_2^2}{2} - \int_{\mathbb{R}^d} w u_{n-1}^\tau\ dx\right\}, \label{ph02a} \\
	v_n^\tau & = \argmin_{v\in W^{1,2}(\mathbb{R}^d)} \left\{ \frac{\big\|v-v_{n-1}^\tau\big\|_2^2}{2\tau} + \frac{\|\nabla v\|_2^2 + \|v\|_2^2}{2} - \int_{\mathbb{R}^d}  v w_{n}^\tau\ dx\right\}, \label{ph02b} \\
	u_n^\tau & = \argmin_{u\in \mathcal{U}_m} \left\{ \frac{\mathcal{W}_2^2\big(u,u_{n-1}^\tau\big)}{2\tau} + \mathcal{E}\big[u,v_n^\tau\big]\right\},  \label{ph02c} 
 \end{align}
with
\begin{equation}
	\mathcal{E}[u,v] := \frac{M^{m-1}}{m-1} \|u\|_m^m - M \int_{\mathbb{R}^d} uv\ dx, \label{ph02d}
\end{equation}
\end{subequations}
are well-defined, where $\mathcal{W}_2$ denotes the $2$-Wasserstein distance on $\mathcal{P}_2(\mathbb{R}^d)$. In addition, for each $n\ge 1$,
\begin{equation*}
	\big( u_n^\tau, v_n^\tau, w_n^\tau \big) \in L^2(\mathbb{R}^d)\times W^{4,2}(\mathbb{R}^d)\times W^{2,2}(\mathbb{R}^d), 
\end{equation*}
and it follows from~\eqref{ph02b} and~\eqref{ph02a} that $v_n^\tau$ and $w_n^\tau$ solve
\begin{subequations}\label{ph04}
\begin{align}
	\frac{v_n^\tau - v_{n-1}^\tau}{\tau} - \Delta v_n^\tau + v_n^\tau = w_n^\tau \;\;\text{ in }\;\; \mathbb{R}^d, \label{ph04a} \\
	\frac{w_n^\tau - w_{n-1}^\tau}{\tau} - \Delta w_n^\tau + w_n^\tau = u_{n-1}^\tau \;\;\text{ in }\;\; \mathbb{R}^d, \label{ph04b}
\end{align}
\end{subequations}
see \cite[Proposition~2.2]{Mim2024b}. 

Since $\tau\in (0,1)$ is fixed throughout this section, we drop from now on the superscript $\tau$ from the notation $(u_n^\tau,v_n^\tau,w_n^\tau)_{n\ge 1}$. We now turn to the derivation of estimates on $(u_n,v_n,w_n)_{n\ge 1}$ and begin with an estimate in $L^m(\mathbb{R}^d)\times W^{2,2}(\mathbb{R}^d)\times L^2(\mathbb{R}^d)$. 

\subsection{A Liapunov functional}\label{sec.2.1}

As a preliminary step, we study the behavior of the Liapunov functional
\begin{equation}
	\mathcal{L}[u,v,w] := \frac{\mathcal{E}[u,v]}{M} + \frac{\|v-\Delta v\|_2^2}{2} + \frac{\|w-v+\Delta v\|_2^2}{2}, \label{ph08}
\end{equation}
which is well-defined for $(u,v,w)\in X_0$.

\begin{lemma}\label{lemds01}
For $n\ge 1$, 
\begin{equation}
\begin{split}
	\mathcal{L}\big[u_n,v_n,w_n\big] + \mathcal{D}_0\big[u_n,u_{n-1},v_n,v_{n-1}\big] & + \mathcal{D}_1\big[v_n,v_{n-1},w_n,w_{n-1}\big] \\
	& \le  \mathcal{L}\big(u_{n-1},v_{n-1},w_{n-1}\big),
\end{split}\label{ph09}
\end{equation}
where
\begin{align*}
	\mathcal{D}_0\big[u_n,u_{n-1},v_n,v_{n-1}\big] & =  \frac{\mathcal{W}_2^2\big(u_n,u_{n-1}\big)}{2M\tau} + 2 \frac{\|\nabla(v_n-v_{n-1})\|_2^2}{\tau} + 2 \frac{\|v_n-v_{n-1}\|_2^2}{\tau} \ge 0, \\
	\mathcal{D}_1\big[v_n,v_{n-1},w_n,w_{n-1}\big] & = \frac{\|w_n-w_{n-1} - v_n + v_{n-1} + \Delta v_n - \Delta v_{n-1}\|_2^2}{2} \\
	& \hspace{3cm} + \frac{\|v_n - v_{n-1} - \Delta v_n + \Delta v_{n-1}\|_2^2}{2} \ge 0.
\end{align*}
\end{lemma}

\begin{proof}
We proceed along the lines of \cite[Proposition~3.1]{Mim2024b} and first infer from~\eqref{ph02c} (with $u=u_{n-1}$) that
\begin{equation*}
	\frac{\mathcal{W}_2^2\big(u_n,u_{n-1}\big)}{2\tau} + \mathcal{E}[u_n,v_n] \le \mathcal{E}[u_{n-1},v_n],
\end{equation*}
from which we deduce that
\begin{equation}
	 \mathcal{E}[u_n,v_n] + \frac{\mathcal{W}_2^2\big(u_n,u_{n-1}\big)}{2\tau} \le \mathcal{E}[u_{n-1},v_{n-1}] - M \int_{\mathbb{R}^d} u_{n-1} (v_n-v_{n-1})\ dx. \label{ph10}
\end{equation}
Next, by~\eqref{ph04b},
\begin{equation}
	- \int_{\mathbb{R}^d} u_{n-1} (v_n-v_{n-1})\ dx =- \int_{\mathbb{R}^d} (v_n-v_{n-1}) \left( \frac{w_n-w_{n-1}}{\tau} - \Delta w_{n} + w_n \right)\ dx. \label{ph11}
\end{equation}
It first follows from~\eqref{ph04a} that
\begin{align*}
	- \int_{\mathbb{R}^d} w_{n} (v_n-v_{n-1})\ dx & = - \frac{\|v_n-v_{n-1}\|_2^2}{\tau} - \int_{\mathbb{R}^d} \nabla(v_n-v_{n-1})\cdot \nabla v_n\ dx \\
	& \qquad - \int_{\mathbb{R}^d} (v_n-v_{n-1}) v_n\ dx,
\end{align*}
whence, since $\mathbf{Y}\cdot (\mathbf{Y}-\mathbf{X}) = \big[ |\mathbf{Y}|^2 - |\mathbf{X}|^2 + |\mathbf{Y}-\mathbf{X}|^2 \big]/2$,
\begin{equation}
\begin{split}
	- \int_{\mathbb{R}^d} w_{n} (v_n-v_{n-1})\ dx & = - \frac{\|v_n-v_{n-1}\|_2^2}{\tau} - \frac{\|\nabla v_n\|_2^2 + \|v_n\|_2^2}{2} \\
	& \qquad + \frac{\|\nabla v_{n-1}\|_2^2 + \|v_{n-1}\|_2^2}{2} - \frac{\|\nabla(v_n-v_{n-1})\|_2^2}{2} \\
	& \qquad - \frac{\|v_n-v_{n-1}\|_2^2}{2}. 
\end{split}\label{ph12}
\end{equation}
Similarly, by~\eqref{ph04a},
\begin{align*}
	\int_{\mathbb{R}^d} \Delta w_{n} (v_n-v_{n-1})\ dx & = - \int_{\mathbb{R}^d} \nabla w_{n} \cdot \nabla (v_n-v_{n-1})\ dx \\
	& = - \int_{\mathbb{R}^d} \left( \frac{\nabla (v_n-v_{n-1})}{\tau} - \nabla\Delta v_n + \nabla v_n \right) \cdot \nabla (v_n-v_{n-1})\ dx \\
	& = -  \frac{\|\nabla(v_n-v_{n-1})\|_2^2}{\tau} - \int_{\mathbb{R}^d} \Delta v_n \Delta(v_n-v_{n-1})\ dx \\
	& \qquad - \int_{\mathbb{R}^d} \nabla v_n \cdot \nabla(v_n-v_{n-1})\ dx,
\end{align*}
so that
\begin{equation}
\begin{split}
	\int_{\mathbb{R}^d} \Delta w_{n} (v_n-v_{n-1})\ dx & = -  \frac{\|\nabla(v_n-v_{n-1})\|_2^2}{\tau} - \frac{\|\Delta v_n\|_2^2 + \|\nabla v_n\|_2^2}{2} \\
	& \qquad + \frac{\|\Delta v_{n-1}\|_2^2 + \|\nabla v_{n-1}\|_2^2}{2} - \frac{\|\Delta(v_n-v_{n-1})\|_2^2}{2} \\
	& \qquad - \frac{\|\nabla(v_n-v_{n-1})\|_2^2}{2}. 
\end{split}\label{ph13}
\end{equation}
Finally, setting $v_{-1} := (1+\tau)v_0 - \tau \Delta v_0 - \tau w_0$ and using again~\eqref{ph04a}, we obtain
\begin{align*}
	& - \int_{\mathbb{R}^d} \frac{w_n-w_{n-1}}{\tau} (v_n-v_{n-1})\ dx \\
	& \quad = - \frac{1}{\tau} \int_{\mathbb{R}^d} \left[ \frac{v_n-v_{n-1}}{\tau} - \Delta v_n + v_n - \frac{v_{n-1}-v_{n-2}}{\tau}  +\Delta v_{n-1} -v_{n-1} \right] (v_n-v_{n-1})\ dx \\
	& \quad = - \frac{1}{\tau} \int_{\mathbb{R}^d} \frac{(v_n-v_{n-1}) - (v_{n-1}-v_{n-2})}{\tau} (v_n-v_{n-1})\ dx \\
	& \quad\qquad - \frac{\|\nabla(v_n-v_{n-1})\|_2^2}{\tau} - \frac{\|v_n-v_{n-1}\|_2^2}{\tau},
\end{align*}
whence
\begin{equation}
\begin{split}
	& - \int_{\mathbb{R}^d} \frac{w_n-w_{n-1}}{\tau} (v_n-v_{n-1})\ dx \\
	& \quad = - \frac{\|v_n-v_{n-1}\|_2^2}{2\tau^2} + \frac{\|v_{n-1}-v_{n-2}\|_2^2}{2\tau^2} - \frac{\|v_n+v_{n-2}-2v_{n-1}\|_2^2}{2\tau^2} \\
	& \qquad - \frac{\|\nabla(v_n-v_{n-1})\|_2^2}{\tau} - \frac{\|v_n-v_{n-1}\|_2^2}{\tau}. 
\end{split}\label{ph14}
\end{equation}
Collecting~\eqref{ph10}, \eqref{ph11}, \eqref{ph12}, \eqref{ph13} and~\eqref{ph14}, we find
\begin{align*}
	\frac{\mathcal{E}[u_n,v_n]}{M} + \frac{\mathcal{W}_2^2(u_n,u_{n-1})}{2M\tau} & \le \frac{\mathcal{E}[u_{n-1},v_{n-1}]}{M} - \frac{\|v_n-v_{n-1}\|_2^2}{2\tau^2} + \frac{\|v_{n-1}-v_{n-2}\|_2^2}{2\tau^2} \\ 
	& \qquad - \frac{\|v_n+v_{n-2}-2v_{n-1}\|_2^2}{2\tau^2} - 2 \frac{\|\nabla(v_n-v_{n-1})\|_2^2}{\tau} \\
	& \qquad - 2 \frac{\|v_n-v_{n-1}\|_2^2}{\tau} - \frac{\|\Delta v_n\|_2^2 + \|\nabla v_n\|_2^2}{2} \\
	& \qquad + \frac{\|\Delta v_{n-1}\|_2^2 + \|\nabla v_{n-1}\|_2^2}{2} - \frac{\|\Delta(v_n-v_{n-1})\|_2^2}{2} \\
	& \qquad - \|\nabla(v_n-v_{n-1})\|_2^2 - \frac{\|\nabla v_n\|_2^2 + \|v_n\|_2^2}{2} \\
	& \qquad + \frac{\|\nabla v_{n-1}\|_2^2 + \|v_{n-1}\|_2^2}{2} - \frac{\|v_n-v_{n-1}\|_2^2}{2}.
\end{align*}
Since
\begin{equation*}
	- \frac{\|\Delta v_n\|_2^2 + \|\nabla v_n\|_2^2}{2} - \frac{\|\nabla v_n\|_2^2 + \|v_n\|_2^2}{2} = - \frac{\|v_n-\Delta v_n\|_2^2}{2}
\end{equation*}
and
\begin{equation*}
	\frac{v_n-v_{n-1}}{\tau} = \Delta v_n - v_n + w_n,
\end{equation*}
we further obtain
\begin{align*}
	\mathcal{L}[u_n,v_n,w_n] & + \mathcal{D}_0\big[u_n,u_{n-1},v_n,v_{n-1}\big] \\
	& \le \mathcal{L}[u_{n-1},v_{n-1},w_{n-1}] - \frac{\|v_n+v_{n-2}-2v_{n-1}\|_2^2}{2\tau^2} - \frac{\|\Delta(v_n-v_{n-1})\|_2^2}{2} \\ 
	& \qquad - \|\nabla(v_n-v_{n-1})\|_2^2 - \frac{\|v_n-v_{n-1}\|_2^2}{2} \\	
	& = \mathcal{L}[u_{n-1},v_{n-1},w_{n-1}] - \frac{\|v_n+v_{n-2}-2v_{n-1}\|_2^2}{2\tau^2} - \frac{\|v_n-v_{n-1} - \Delta(v_n-v_{n-1})\|_2^2}{2} \\
	& = \mathcal{L}[u_{n-1},v_{n-1},w_{n-1}] - \mathcal{D}_1\big[v_n,v_{n-1},w_n,w_{n-1}\big],
\end{align*}
and the proof is complete.
\end{proof}

\subsection{A lower bound for the Liapunov functional}\label{sec.2.2}

The next step is to derive a lower bound on~$\mathcal{L}$ according to the range of $m>1$.

\begin{lemma}\label{lemds03}
Let $d\ge 5$ and $m>1$. 
\begin{subequations}\label{ph15}
\begin{itemize}
	\item [\textbf{(m1)}] if $m>m^* = (2d-4)/d$, then there are $\delta_m>0$ and $b\ge 0$ depending only on $d$, $m$ and $M$ such that
	\begin{equation}
		\mathcal{L}[u,v,w] \ge \delta_m \left( \|u\|_m^m + \|v\|_{W^{1,2}}^2 + \|w-v+\Delta v\|_2^2 \right) - b \label{ph15a}
	\end{equation}
	for all $(u,v,w)\in X_0$;
	\item [\textbf{(m2)}] if $m=m^*$ and $M<M^*$, with $M^*>0$ defined in~\eqref{ph19ba}, then there is $\delta_{m^*}>0$ depending only on $d$ and $M$ such that 
	\begin{equation}
		\mathcal{L}[u,v,w] \ge \delta_{m^*} \left( \|u\|_{m^*}^{m^*} + \|v\|_{W^{1,2}}^2 + \|w-v+\Delta v\|_2^2 \right) \label{ph15b}
	\end{equation}
	for all $(u,v,w)\in X_0$;
	\item [\textbf{(m3)}] if $m\in (m_*,m^*)$, then we define $z_m>0$ depending only on $d$, $m$ and $M$ by~\eqref{ph19c}. For $\theta\in (0,1)$, there is $\delta_{m,\theta}>0$ depending only on $d$, $m$, $M$ and $\theta$ such that 
	\begin{equation}
		\mathcal{L}[u,v,w] \ge \delta_{m,\theta} \left( \|u\|_{m}^{m} + \|v\|_{W^{1,2}}^2 + \|w-v+\Delta v\|_2^2 \right) \label{ph15c}
	\end{equation}
	for all $(u,v,w)\in X_0$ satisfying additionally $\|u\|_m^m \le \theta z_m$.
\end{itemize}
\end{subequations}
\end{lemma}

\begin{proof}
Let $d\ge 5$ and $m>m_*$. We infer from H\"older's inequality and the continuous embedding of $\dot{W}^{2,2}(\mathbb{R}^d)$ in $L^{2d/(d-4)}(\mathbb{R}^d)$ that
\begin{align*}
	\left| \int_{\mathbb{R}^d} uv\ dx \right| & \le \|u\|_{2d/(d+4)} \|v\|_{2d/(d-4)} \\
	& \le C \|u\|_1^{[m(d+4)-2d]/[2d(m-1)]} \|u\|_m^{m(d-4)/[2d(m-1)]} \|v\|_{\dot{W}^{2,2}} \\
	& \le C \|u\|_m^{m(d-4)/[2d(m-1)]} \|v\|_{\dot{W}^{2,2}},
\end{align*}
from which we deduce that
\begin{equation}
	K_m := \sup_{(u,v)\in Y_0} \left\{ \frac{\displaystyle{\int_{\mathbb{R}^d} uv\ dx}}{\|u\|_m^{m(d-4)/[2d(m-1)]}\|v\|_{\dot{W}^{2,2}} } \right\}<\infty. \label{ph16}
\end{equation}
As in \cite[Section~4]{Mim2024b}, we next observe that, for $u\in (L_+^1\cap L^m)(\mathbb{R}^d)$ satisfying $\|u\|_1=1$, the function $(-\Delta)^{-2}u$ belongs to $\dot{W}^{2,2}(\mathbb{R}^d)$ with
\begin{equation*}
	\|\Delta (-\Delta)^{-2} u\|_2^2 = \int_{\mathbb{R}^d} (-\Delta)^{-2} u [\Delta^2 (-\Delta)^{-2} u]\ dx = \int_{\mathbb{R}^d} u (-\Delta)^{-2} u \ dx, 
\end{equation*}
so that
\begin{equation*}
	K_m^2 \ge \frac{\displaystyle{\left( \int_{\mathbb{R}^d} u (-\Delta)^{-2}u \ dx \right)^2}}{\|u\|_m^{m(d-4)/[d(m-1)]} \|(-\Delta)^{-2}u\|^2_{\dot{W}^{2,2}}} = \frac{\displaystyle{\int_{\mathbb{R}^d} u (-\Delta)^{-2}u \ dx}}{\|u\|_m^{m(d-4)/[d(m-1)]}}.
\end{equation*}
Consequently,
\begin{equation*}
	K_m^2 \ge \sup_{u\in (L_+^1\cap L^m)(\mathbb{R}^d), \ \|u\|_1=1} \frac{\displaystyle{\int_{\mathbb{R}^d} u (-\Delta)^{-2}u \ dx}}{\|u\|_m^{m(d-4)/[d(m-1)]}}.
\end{equation*}
Conversely, for $(u,v)\in Y_0$, 
\begin{align*}
	\left| \int_{\mathbb{R}^d} uv\ dx \right| & = \left| \int_{\mathbb{R}^d} u [(-\Delta)^{-1}\Delta v]\ dx \right| = \left| \int_{\mathbb{R}^d} (-\Delta)^{-1}u \Delta v\ dx \right| \\
	& \le \|(-\Delta)^{-1}u\|_2 \|\Delta v\|_2 = \left( \int_{\mathbb{R}^d} u (-\Delta)^{-2}u \ dx \right)^{1/2} \|v\|_{\dot{W}^{2,2}},
\end{align*}
from which we deduce that
\begin{equation*}
	K_m^2 \le \sup_{u\in (L_+^1\cap L^m)(\mathbb{R}^d), \ \|u\|_1=1} \frac{\displaystyle{\int_{\mathbb{R}^d} u (-\Delta)^{-2}u \ dx}}{\|u\|_m^{m(d-4)/[d(m-1)]}}.
\end{equation*}
Therefore,
\begin{equation}
	K_m^2 = \sup_{u\in (L_+^1\cap L^m)(\mathbb{R}^d), \ \|u\|_1=1} \frac{\displaystyle{\int_{\mathbb{R}^d} u (-\Delta)^{-2}u \ dx}}{\|u\|_m^{m(d-4)/[d(m-1)]}}. \label{ph17}
\end{equation}
Next, since
\begin{equation*}
	\mathcal{L}[u,v,w] = \frac{\mathcal{E}[u,v]}{M} + \frac{\|\Delta v\|_2^2}{2} + \frac{\|w-v+\Delta v\|_2^2}{2} + \|\nabla v\|_2^2 + \frac{\|v\|_2^2}{2}
\end{equation*}
and
\begin{align*}
	\frac{\mathcal{E}[u,v]}{M} & + \frac{\|\Delta v\|_2^2}{2} - \frac{\mathcal{E}[u,(-\Delta)^{-2}u]}{M} - \frac{\|\Delta (-\Delta)^{-2}u\|_2^2}{2} \\
	& = - \int_{\mathbb{R}^d} u \big[v-(-\Delta)^{-2}u\big]\ dx + \int_{\mathbb{R}^d} [\Delta(-\Delta)^{-2}u] \Delta[v-(-\Delta)^{-2}u]\ dx \\
	& \qquad + \frac{\|\Delta[v-(-\Delta)^{-2}u]\|_2^2}{2} \\
	& = \frac{\|\Delta[v-(-\Delta)^{-2}u]\|_2^2}{2},
\end{align*}
we infer from~\eqref{ph17} that
\begin{align*}
	\mathcal{L}[u,v,w] & \ge \frac{\mathcal{E}[u,(-\Delta)^{-2}u]}{M} + \frac{\|\Delta (-\Delta)^{-2}u\|_2^2}{2} + \frac{\|w-v+\Delta v\|_2^2}{2} + \|\nabla v\|_2^2 + \frac{\|v\|_2^2}{2} \\
	& = \frac{M^{m-2}}{m-1} \|u\|_m^m - \frac{1}{2} \int_{\mathbb{R}^d} u (-\Delta)^{-2} u\ dx + \frac{\|w-v+\Delta v\|_2^2}{2} + \|\nabla v\|_2^2 + \frac{\|v\|_2^2}{2} \\
	& \ge \frac{M^{m-2}}{m-1} \|u\|_m^m - \frac{K_m^2}{2} \|u\|_m^{m(d-4)/[d(m-1)]} + \frac{\|w-v+\Delta v\|_2^2}{2} + \|\nabla v\|_2^2 + \frac{\|v\|_2^2}{2},
\end{align*}
that is, introducing
\begin{equation*}
	f_m(z) := \frac{M^{m-2}}{m-1} z - \frac{K_m^2}{2} z^{(d-4)/[d(m-1)]}, \qquad z\in (0,\infty),
\end{equation*}
\begin{equation}
	\mathcal{L}[u,v,w] \ge f_m\big(\|u\|_m^m\big) + \frac{\|w-v+\Delta v\|_2^2}{2} + \|\nabla v\|_2^2 + \frac{\|v\|_2^2}{2}. \label{ph18}
\end{equation}
At this point, we note that
\begin{equation*}
	\frac{d-4}{d(m-1)} = 1 + \frac{m^*-m}{m-1}.
\end{equation*}
Consequently,
\begin{subequations}\label{ph19}
\begin{itemize}
	\item [$\mathbf{(m1)}$] If $m>m^*$, then $(d-4)/[d(m-1)]<1$ and it follows from Young's inequality that
	\begin{equation}
		f_m(z) \ge \delta_m z - C M^{-[(m-2)(d-4)]/[d(m-m^*)]}, \qquad z\in (0,\infty), \label{ph19a}
	\end{equation}
	with $\delta_m := [(m-m^*) M^{m-2}]/(m-1)^2>0$.
	\item [$\mathbf{(m2)}$] If $m=m^*$, then $(d-4)/[d(m-1)]=1$ and 
	\begin{equation}
		f_{m^*}(z) = \left( \frac{M^{m^*-2}}{m^*-1} - \frac{K_{m^*}^2}{2} \right) z = \delta_{m^*} z \ge 0, \qquad z\in (0,\infty), \label{ph19b}
	\end{equation}
	with $\delta_{m^*} := (2M^{m^*-2}-(m^*-1)K_{m^*}^2)/[2(m^*-1)]>0$ provided 
	\begin{equation}
		M^{m^*-2}>\big(M^*\big)^{m^*-2} := (m^*-1)K_{m^*}^2/2. \label{ph19ba}
	\end{equation} 
	Equivalently, recalling that $m^*<2$, the lower bound~\eqref{ph19b} holds true for $M<M^*$.
	\item [$\mathbf{(m3)}$] If $m\in(m_*,m^*)$, then $(d-4)/[d(m-1)]>1$ and $f_m$ has a unique maximum $z_m$ given by
	\begin{equation}
		z_m := \left[ \frac{2dM^{m-2}}{(d-4)K_m^2} \right]^{(m-1)/(m^*-m)}, \quad f_m(z_m) = \frac{d(m^*-m)M^{m-2}}{(d-4)(m-1)}z_m. \label{ph19c}
	\end{equation}
	Therefore, for $\theta\in (0,1)$ and $z\in (0,\theta z_m)$,
	\begin{align*}
		f_m(z) & = \left[ \frac{M^{m-2}}{m-1} - \frac{K_m^2}{2} z_m^{(m^*-m)/(m-1)} \left( \frac{z}{z_m} \right)^{(m^*-m)/(m-1)} \right] z \\
		& \ge \delta_{m,\theta} z
	\end{align*}
	with
	\begin{equation}
		\delta_{m,\theta} := \frac{M^{m-2}}{(d-4)(m-1)} \left[ d-4 - d(m-1) \theta^{(m^*-m)/(m-1)} \right] >0. \label{ph19ca}
	\end{equation}
\end{itemize}
\end{subequations}
We now combine~\eqref{ph18} with the outcome of~\eqref{ph19} to complete the proof of Lemma~\ref{lemds03}.
\end{proof}

Subsequently, we also establish a lower bound on $\mathcal{L}$ in the lower dimensional cases $1\le d\le 4$ for $m>1$.

\begin{lemma}
Let $1\le d \le 4$ and $m>1$.
\begin{itemize}
	\item[\textbf{(m4)}] There are $\delta_m>0$ and $b\ge 0$ depending only on $d$, $m$ and $M$ such that
	\begin{equation}
	\mathcal{L}[u,v,w] \ge \delta_m \left( \|u\|_m^m + \|(I-\Delta)v\|_{2}^2 + \|w-v+\Delta v\|_2^2 \right) - b \label{ph15d}
	\end{equation}
	for all $(u,v,w)\in X_0$.
\end{itemize}
\end{lemma}

\begin{proof}
Let $1\le d \le 4$ and $m>1$. First notice that it follows from the Sobolev duality  and Young's inequality that
\begin{align*}
\left|\int_{\mathbb{R}^d}uv\,dx\right|
\le\,&\|(I-\Delta)^{-1}u\|_2 \| (I-\Delta)v\|_{2}
\\
\le\,& M \|(I-\Delta)^{-1}u\|_2^2 + \frac{1}{4M} \|(I-\Delta)v\|_{2}^2.
\end{align*}
Next, owing to $m>1$, H\"older's inequality and  the boundedness of the Bessel potential operator $(I-\Delta)^{-2}$ from $L^1(\mathbb{R}^d)$ to~$L^p(\mathbb{R}^d)$ with $p\in[1,\infty]$ if $d\in\{1,2,3\}$ and $p\in[1,\infty)$ if $d=4$ imply that
\begin{align*}
\|(I-\Delta)^{-1}u\|_2^2
=\,&\int_{\mathbb{R}^d} u (I-\Delta)^{-2}u \,dx
\\
\le\,& \|(I-\Delta)^{-2}u\|_{m/(m-1)}\|u\|_m 
\\
\le\,& C\|u\|_1 \|u\|_m \le\,C \|u\|_m.
\end{align*}
Consequently, 
\begin{equation*}
	\mathcal{L}[u,v,w] \ge\, \frac{M^{m-1}}{m-1} \|u\|_m^m + \frac{\|v-\Delta v\|_2^2}{2} + \frac{\|w-v+\Delta v\|_2^2}{2} - CM \|u\|_m - \frac{\|v-\Delta v\|_2^2}{4}  
\end{equation*}
and the desired bound on $\mathcal{L}$ 
\begin{equation*}
	\mathcal{L}[u,v,w]\ge\,\delta_{m}( \|u\|_m^m+\|v-\Delta v\|_2^2+\|w-v+\Delta v\|_2^2 )-b
\end{equation*}
for some $\delta_{m}>0$ and $b>0$ follows from Young's inequality.
\end{proof}

\subsection{$L^m$-estimate on $(u_n)_{n\ge 1}$}\label{sec.2.3}

We now exploit Lemma~\ref{lemds01} and Lemma~\ref{lemds03} to derive a bound on $(u_n,v_n,w_n)_{n\ge 1}$ in $L^m(\mathbb{R}^d)\times W^{2,2}(\mathbb{R}^d)\times L^2(\mathbb{R}^d)$.

\begin{corollary}\label{cords04}
Let $d\ge1$ and $m>1$. 
\begin{subequations}\label{ph20}
\begin{itemize}
	\item [\textbf{(m1)}] If $d\ge 5$ and $m>m^*$, then there is $A>0$ depending only on $d$, $m$, $M$ and $(u_0,v_0,w_0)$ such that
	\begin{equation}
		\|u_n\|_m^m + \|v_n-\Delta v_n\|_2^2 + \|w_n\|_2^2 \le A, \qquad n\ge 1. \label{ph20a}
	\end{equation}
	\item [\textbf{(m2)}] If $d\ge 5$, $m=m^*$ and $M\in \big(0,M^*\big)$, then there is $A>0$ depending only on $d$, $M$ and $(u_0,v_0,w_0)$ such that
	\begin{equation}
		\|u_n\|_{m^*}^{m^*} + \|v_n-\Delta v_n\|_2^2 + \|w_n\|_2^2 \le A, \qquad n\ge 1. \label{ph20b}
	\end{equation}
	\item [\textbf{(m3)}] If $d\ge 5$ and $m\in \big(m_*,m^*\big)$ and $(u_0,v_0,w_0)$ are such that 
	\begin{equation*}
		\|u_0\|_m^m\le \theta z_m \;\;\text{ and }\;\; \mathcal{L}[u_0,v_0,w_0]\le f_m(\theta z_m)
	\end{equation*} 
	for some $\theta\in (0,1)$, then 
	\begin{equation}
		\|u_n\|_m^m \le \theta z_m, \qquad n\ge 1, \label{ph20c}
	\end{equation}
	and there is $A>0$ depending only on $d$, $m$, $M$, $\theta$ and $(u_0,v_0,w_0)$ such that
	\begin{equation}
		\|v_n-\Delta v_n\|_2^2 + \|w_n\|_2^2 \le A, \qquad n\ge 1. \label{ph20d}
	\end{equation}
	\item [\textbf{(m4)}] If $1\le d\le 4$ and $m>1$, then there is $A>0$ depending only on $d$, $m$, $M$ and $(u_0,v_0,w_0)$ such that
		\begin{equation}
		\|u_n\|_m^m + \|v_n-\Delta v_n\|_2^2 + \|w_n\|_2^2 \le A, \qquad n\ge 1. \label{ph20e}
		\end{equation}
\end{itemize}
\end{subequations}
\end{corollary}

\begin{proof} Let $n\ge 1$. By~\eqref{ph09} in Lemma~\ref{lemds01}, 
\begin{equation}
	\mathcal{L}[u_n,v_n,w_n] \le \mathcal{L}[u_{n-1},v_{n-1},w_{n-1}] \le \mathcal{L}[u_0,v_0,w_0]. \label{ph21}
\end{equation}

\medskip

\noindent\textbf{(m1)} It readily follows from~\eqref{ph15a} and~\eqref{ph21} that
\begin{equation*}
	\|u_n\|_{m}^{m} + \|v_n\|_{W^{1,2}}^2 + \|w_n-v_n+\Delta v_n\|_2^2 \le \frac{\mathcal{L}[u_n,v_n,w_n]+b}{\delta_m} \le \frac{\mathcal{L}[u_0,v_0,w_0]+b}{\delta_m}.
\end{equation*}
In addition, we infer from~\eqref{ph08}, \eqref{ph16} and~\eqref{ph21} that
\begin{align*}
	\|v_n-\Delta v_n\|_2^2 & \le 2 \mathcal{L}[u_n,v_n,w_n] + 2 \int_{\mathbb{R}^d} u_n v_n\ dx \\
	& \le 2 \mathcal{L}[u_0,v_0,w_0] + 2 K_m \|u_n\|_m^{m(d-4)/[2d(m-1)]}, 
\end{align*}
and~\eqref{ph20a} is an immediate consequence of the above two inequalities.

\medskip

\noindent\textbf{(m2)} In that case, owing to the assumption $M\in\big(0,M^*\big)$, the proof of~\eqref{ph20b} relies on a similar argument based on~\eqref{ph15b} instead of~\eqref{ph15a}.

\medskip

\noindent\textbf{(m3)} We pick $\theta \in (0, 1)$ such that
\begin{equation}
	\|u_0\|_m^m\le \theta z_m \;\;\text{and}\;\; \mathcal{L}[u_0,v_0,w_0] \le f(\theta z_m). \label{ph22}
\end{equation} 
Then, by~\eqref{ph18}, \eqref{ph21} and~\eqref{ph22},
\begin{equation*}
	f_m(\|u_n\|_m^m) \le \mathcal{L}[u_n,v_n,w_n] \le \mathcal{L}[u_0,v_0,w_0] < f_m(\theta z_m),
\end{equation*}
from which we deduce~\eqref{ph20c} due to the monotonicity of $f_m$ on $[0,z_m]$, since $u_n\in\mathcal{U}_m$ guarantees that $\|u_n\|_m^m \in [0,z_m]$. We then proceed as in the previous cases with the help of~\eqref{ph15c} to complete the proof of~\eqref{ph20d}.

\medskip

\noindent\textbf{(m4)} The proof of~\eqref{ph20e} is based on~\eqref{ph15d} and is similar to that of~\eqref{ph20a}.
\end{proof}

\subsection{$L^\rho$-estimate on $(u_n)_{n\ge 1}$, $\rho\in (m,\infty)$}\label{sec.2.4}

In the next step, we employ the flow interchange method introduced in \cite{MMS2009} to derive $L^\rho$-bounds on $(u_n)_{n\ge 1}$ for $\rho\in (m,\infty)$ in terms of $(v_n-\Delta v_n)_{n\ge 1}$. To estimate the latter, we take advantage of the just established $L^m$-estimate on $(u_n)_{n\ge 1}$ to derive additional estimates on $(w_n)_{n\ge 1}$ and $(v_n)_{n\ge 1}$ using the discrete heat equations~\eqref{ph04}, the keystone in that direction being Lemma~\ref{lemds02a} stated in the Appendix.

\begin{lemma}\label{lemds04}
Let $d\ge 1$ and $m>1$ satisfying additionally $m>m_*$ when $d\ge 5$ and assume that there is $B>0$ such that
\begin{equation}
	\|u_n\|_m \le B, \qquad n\ge 1. \label{ph23} 
\end{equation}
Then, for any $\rho > \max\{m,(d-2)/2\}$, there are $q_0>1$ 
depending only on $d$, $m$ and $\rho$ and $B_\rho>0$ depending only on $d$, $m$, $(u_0,v_0,w_0)$, $B$ and $\rho$ such that
\begin{equation}
\|u_N\|_\rho^\rho + \tau \sum_{n=1}^N (1+\tau)^{-q_0(N-n)}\big\|\nabla u_n^{(m+\rho-1)/2}\big\|_2^2  \le B_\rho, \qquad N\ge 1. \label{ph24}
\end{equation}
\end{lemma}

\begin{remark}
Time weighted estimates, combined with maximal regularity, are used in \cite[Proposition~3.2]{IsYo2020} as an intermediate step to establish $L^\rho$-boundedness for $u$ for $\rho\in (1,\infty)$. Taking into account that
\begin{equation*}
(1+\tau)^{-q_0 (N-n)}\approx e^{-q_0\tau(N-n)}\quad \text{as}~~\tau \to 0,
\end{equation*}
the second term on the left-hand side of~\eqref{ph24} reads
\begin{align*}
\tau \sum_{n=1}^N (1+\tau)^{-q_0(N-n)}\big\|\nabla u_n^{(m+\rho-1)/2}\big\|_2^2
\approx \int_0^t e^{-q_0(t-s)} \big\|\nabla u(s)^{(m+\rho-1)/2}\big\|_2^2\,ds
\end{align*}
at the continuous level and matches \cite[Equation~(36)]{IsYo2020}. We develop a discrete analogue of this approach in the proof of Lemma~\ref{lemds04}.
\end{remark}

\begin{proof} Let $\rho>m$. We begin with the derivation of an estimate on $\|u_n\|_\rho$ and $\big\|\nabla u_n^{(m+\rho-1)/2}\big\|_2$ in terms of $\|u_{n-1}\|_\rho$ and $\big\|u_n^\rho (v_n-\Delta v_n)\big\|_{1}$. This step is achieved with the help of the flow interchange method \cite{MMS2009}. In a second step, we use a variant of the Gagliardo-Nirenberg inequality to improve the previous estimate to a control of $\|u_n\|_\rho$ and $\big\|\nabla u_n^{(m+\rho-1)/2}\big\|_2$ in terms of $\|u_{n-1}\|_\rho$ and $\|v_n-\Delta v_n\|_{r/(r-1)}$ for suitably chosen values of $r\in (1,\infty)$. We next use Lemma~\ref{lemds02a} to estimate $\|v_n-\Delta v_n\|_{r/(r-1)}$ in terms of suitable norms involving $(u_n)_{n\ge 1}$.
	
\medskip 
	
\noindent\textbf{Step~1.} Let $n\ge 1$ and let $U$ be the unique weak solution to the porous medium equation
\begin{equation}
\begin{split}
	\partial_t U & = \Delta U^\rho \;\;\text{ in }\;\; (0,\infty)\times\mathbb{R}^d, \\
	U(0) & = u_n \;\;\text{ in }\;\; \mathbb{R}^d.
\end{split}\label{ph25}
\end{equation}
Classical properties of~\eqref{ph01} entail that $U(t)\in (\mathcal{P}_2\cap L^m)(\mathbb{R}^d)$ for all $t\ge 0$ and that
\begin{equation}
	\|U(t_2)\|_\rho \le \|U(t_1)\|_\rho, \qquad t_2 \ge t_1 \ge 0. \label{ph26}
\end{equation}
Moreover, since $\rho>1>(d-1)/d$, the porous medium equation~\eqref{ph25} can be interpreted as a gradient flow of the functional
\begin{equation*}
	\mathcal{F}[z] := \frac{\|z\|_\rho^\rho}{\rho-1}, \qquad z\in \mathcal{P}_2(\mathbb{R}^d)\cap L^\rho(\mathbb{R}^d),
\end{equation*}
for the Wasserstein distance $\mathcal{W}_2$ in $\mathcal{P}_2(\mathbb{R}^d)$, see \cite[Theorem~11.2.5]{AGS2005} for instance. Moreover, by \cite[Theorem~11.1.4]{AGS2005}, 
\begin{equation*}
	\frac{1}{2} \left[ \mathcal{W}_2^2(U(t),u_{n-1}) - \mathcal{W}_2^2(u_n,u_{n-1})
\right] \le \int_0^t \big(\mathcal{F}[u_{n-1}] - \mathcal{F}[U(s)]\big)\ ds, \qquad t\ge 0,
\end{equation*}
from which we deduce, in view of the time monotonicity property~\eqref{ph26}, that
\begin{equation}
	\frac{1}{2\tau} \left[ \mathcal{W}_2^2(U(t),u_{n-1}) - \mathcal{W}_2^2(u_n,u_{n-1})
	\right] \le \frac{t}{\tau(\rho-1)} \big(\|u_{n-1}\|_{\rho}^{\rho} - \|U(t)\|_\rho^\rho \big), \qquad t\ge 0. \label{ph27}
\end{equation}
We next infer from the definition~\eqref{ph02d} of $\mathcal{E}$,~\eqref{ph25} and the non-negativity of $U$ and $v_n$ that
\begin{align*}
	\frac{d}{dt} \mathcal{E}[U,v_n] & = \frac{mM^{m-1}}{m-1} \int_{\mathbb{R}^d} U^{m-1} \Delta U^\rho\ dx - M \int_{\mathbb{R}^d} v_n \Delta U^\rho\ dx \\
	& = - \frac{4m\rho M^{m-1}}{(m+\rho-1)^2} \big\|\nabla U^{(m+\rho-1)/2}\big\|_2^2 - M\int_{\mathbb{R}^d} U^\rho \Delta v_n\ dx \\
	& \le - \frac{4m\rho M^{m-1}}{(m+\rho-1)^2} \big\|\nabla U^{(m+\rho-1)/2}\big\|_2^2 + M\int_{\mathbb{R}^d} U^\rho (v_n - \Delta v_n)\ dx.
\end{align*}
Integrating with respect to time gives
\begin{equation}
\begin{split}
	\mathcal{E}[U(t),v_n] & \le \mathcal{E}[u_n,v_n] - \frac{4m\rho M^{m-1}}{(m+\rho-1)^2} \int_0^t \big\|\nabla U^{(m+\rho-1)/2}(s)\big\|_2^2\ ds \\
	& \qquad + M\int_0^t \int_{\mathbb{R}^d} U^\rho(s) (v_n - \Delta v_n)\ dxds, \qquad t\ge 0.
\end{split}\label{ph28}
\end{equation}
Now, $t\mapsto \|U(t)\|_m$ is a non-increasing function, so that $\|U(t)\|_m\le \|u_n\|_m\le B$ and thus $U(t)\in\mathcal{U}_m$ for $t\ge 0$. Therefore, the definition~\eqref{ph02c} of $u_n$ implies that
\begin{equation*}
	\frac{1}{2\tau} \mathcal{W}_2^2(u_n,u_{n-1}) + \mathcal{E}[u_n,v_n] \le \frac{1}{2\tau} \mathcal{W}_2^2(U(t),u_{n-1}) + \mathcal{E}[U(t),v_n], \qquad t\ge 0,
\end{equation*}
and we deduce from~\eqref{ph27}, \eqref{ph28} and the above inequality that, for $t\ge 0$, 
\begin{align*}
	\frac{1}{2\tau} \mathcal{W}_2^2(u_n,u_{n-1}) + \mathcal{E}[u_n,v_n] & \le \frac{1}{2\tau} \mathcal{W}_2^2(u_n,u_{n-1}) + \frac{t}{\tau(\rho-1)} \big(\|u_{n-1}\|_{\rho}^{\rho} - \|U(t)\|_\rho^\rho \big)\\
	& \qquad  + \mathcal{E}[u_n,v_n] - \frac{4m\rho M^{m-1}}{(m+\rho-1)^2} \int_0^t \big\|\nabla U^{(m+\rho-1)/2}(s)\big\|_2^2\ ds \\
	& \qquad + M\int_0^t \int_{\mathbb{R}^d} U^\rho(s) (v_n - \Delta v_n)\ dxds.
\end{align*}
Hence, after dividing by $t$,
\begin{align*}
	\frac{4m\rho M^{m-1}}{(m+\rho-1)^2 t} \int_0^t \big\|\nabla U^{(m+\rho-1)/2}(s)\big\|_2^2\ ds & \le \frac{1}{\tau(\rho-1)} \big(\|u_{n-1}\|_{\rho}^{\rho} - \|U(t)\|_\rho^\rho \big) \\
	& \qquad + \frac{M}{t} \int_0^t \int_{\mathbb{R}^d} U^\rho(s) (v_n - \Delta v_n)\ dxds.
\end{align*}
Since $U\in C\big([0,\infty),L^\rho(\mathbb{R}^d)\big)$, we proceed as in the proof of Step~3 in \cite[Proposition~8]{BlLa2013} to take the limit $t\to 0$ in the above inequality and conclude that
\begin{equation}
	\frac{4m\rho M^{m-1}}{(m+\rho-1)^2} \big\|\nabla u_n^{(m+\rho-1)/2}\big\|_2^2 \le \frac{\|u_{n-1}\|_{\rho}^{\rho} - \|u_n\|_\rho^\rho}{\tau(\rho-1)} + M\int_{\mathbb{R}^d} u_n^\rho (v_n - \Delta v_n)\ dx. \label{ph29}
\end{equation}

\medskip 

\noindent\textbf{Step~2.} Now, let $r\in (1,\infty)$ to be specified later satisfying
\begin{equation}
	m \le \rho r \le \frac{d}{d-2} (m+\rho-1) \;\;\text{ for }\;\; d\ge 3, \quad m\le \rho r \;\;\text{ for }\;\; d\in\{1,2\}, \label{ph30}
\end{equation}
and set $r'=r/(r-1)\in (1,\infty)$. By H\"older's inequality,
\begin{equation*}
	\left| \int_{\mathbb{R}^d} u_n^\rho (v_n - \Delta v_n)\ dx \right| \le \|u_n\|_{r\rho}^\rho \|v_n-\Delta v_n\|_{r'}
\end{equation*} 
and it follows from \cite[Lemma~2.4]{Sug2006} (with $q_1=m$, $q_2=r\rho$ and $r=\rho$), \eqref{ph30} and the above inequality that
\begin{equation}
	\left| \int_{\mathbb{R}^d} u_n^\rho (v_n - \Delta v_n)\ dx \right| \le C(r,\rho) \big\|\nabla u_n^{(m+\rho-1)/2}\big\|_2^{2\delta_0} \|u_n\|_m^{(1-\theta_0)\rho} \|v_n-\Delta v_n\|_{r'}, \label{ph31}
\end{equation}
where $\delta_0 := \theta_0\rho/(m+\rho-1)$ and
\begin{align*}
	\frac{2\theta_0}{m+\rho-1} & = \left( \frac{1}{m} - \frac{1}{r\rho} \right) \left( \frac{1}{d} - \frac{1}{2} + \frac{m+\rho-1}{2m} \right)^{-1} \\
	& = \frac{r\rho-m}{mr\rho} \frac{2md}{2m+d(\rho-1)} = \frac{2d(r\rho - m)}{r\rho [2m + d(\rho-1)]}.
\end{align*}
Since
\begin{equation}
	\frac{1}{q_0} := 1 - \delta_0 = 1 - \frac{\theta_0\rho}{m+\rho-1} = \frac{(2m-d)r+md}{r[2m+d(\rho-1)]} \label{delta0}
\end{equation}
is a decreasing function of $r$, we infer from~\eqref{ph30} that
\begin{align*}
	\frac{1}{q_0} = 1 - \delta_0 & \ge \frac{d(2m-d)(m+\rho-1) + md(d-2)\rho}{d(m+\rho-1)[2m+d(\rho-1)]} \\
	& = \frac{(m-1)[d(\rho-1)+2m]}{(m+\rho-1)[2m+d(\rho-1)]} > 0.
\end{align*}
Therefore, $\delta_0\in (0,1)$ (or, equivalently $q_0\in (1,\infty)$) and we may use Young's inequality in~\eqref{ph31}, along with~\eqref{ph23}, to obtain
\begin{align*}
	\left| \int_{\mathbb{R}^d} u_n^\rho (v_n - \Delta v_n)\ dx \right| & \le \frac{2m\rho M^{m-1}}{(m+\rho-1)^2} \big\|\nabla u_n^{(m+\rho-1)/2}\big\|_2^2 \\
	& \quad + C(r,\rho) \|u_n\|_m^{(1-\theta_0)\rho/(1-\delta_0)} \|v_n-\Delta v_n\|_{r'}^{1/(1-\delta_0)} \\
	& \le \frac{2m\rho M^{m-1}}{(m+\rho-1)^2} \big\|\nabla u_n^{(m+\rho-1)/2}\big\|_2^2 + C(r,\rho,B) \|v_n-\Delta v_n\|_{r'}^{q_0}.
\end{align*}
Inserting the above estimate in~\eqref{ph29} leads us to
\begin{equation*}
	\frac{2m\rho M^{m-1}}{(m+\rho-1)^2} \big\|\nabla u_n^{(m+\rho-1)/2}\big\|_2^2 \le \frac{\|u_{n-1}\|_{\rho}^{\rho} - \|u_n\|_\rho^\rho}{\tau(\rho-1)} + C(r,\rho,B) \|v_n-\Delta v_n\|_{r'}^{q_0},
\end{equation*}
recalling that $q_0=1/(1-\delta_0)\in (1,\infty)$ is defined in~\eqref{delta0}. Multiplying both sides of the above inequality by $(1+\tau)^{q_0 n}$ and summing up with respect to $n\in \{1,\dots,N\}$, $N\ge 1$, we find,  
\begin{align*}
	&	\frac{2m\rho M^{m-1}}{(m+\rho-1)^2}\sum_{n=1}^N (1+\tau)^{q_0 n} \big\|\nabla u_n^{(m+\rho-1)/2}\big\|_2^2\\ 
	\le \,& \sum_{n=1}^N (1+\tau)^{q_0 n} \frac{\|u_{n-1}\|_{\rho}^{\rho} - \|u_n\|_\rho^\rho}{\tau(\rho-1)} + C(r,\rho,B) \sum_{n=1}^N (1+\tau)^{q_0 n} \|v_n-\Delta v_n\|_{r'}^{q_0} \\
	=\,&\frac{1}{\tau(\rho-1)} \left[ \sum_{n=0}^{N-1} (1+\tau)^{q_0 (n+1)} \|u_n\|_\rho^\rho - \sum_{n=1}^N (1+\tau)^{q_0 n} \|u_{n}\|_{\rho}^{\rho}  \right] \\
	&+ C(r,\rho,B) \sum_{n=1}^N (1+\tau)^{q_0 n} \|v_n-\Delta v_n\|_{r'}^{q_0} \\
	=\,&\frac{1}{\tau(\rho-1)}\left[ \big[(1+\tau)^{q_0}-1\big] \sum_{n=1}^{N-1} (1+\tau)^{q_0 n} \|u_{n}\|_\rho^\rho + (1+\tau)^{q_0} \|u_{0}\|_{\rho}^{\rho} - (1+\tau)^{q_0 N} \|u_N\|_\rho^\rho \right]
	\\
	& + C(r,\rho,B) \sum_{n=1}^N (1+\tau)^{q_0 n} \|v_n-\Delta v_n\|_{r'}^{q_0}.
	\eqntag
	\label{ph32}
\end{align*}

\medskip 

\noindent\textbf{Step~3.} Let us first consider the case $m\ge d/2$. In order to use the analysis performed in \textbf{Step~2}, we choose $r=m/(m-1)$ (and thus $r'=m$) in that case and first check that this choice of $r$ implies that~\eqref{ph30} is satisfied. This property is actually obvious in that case, since
\begin{equation*}
	\frac{\rho m}{m-1} - m = \frac{m(\rho-m+1)}{m-1}>0
\end{equation*}
and, when $d\ge 3$,
\begin{equation*} 
	\frac{d(\rho+m-1)}{d-2} - \frac{\rho m}{m-1} = \frac{\rho(2m-d)+d(m-1)^2}{(m-1)(d-2)} > 0
\end{equation*}
due to $\rho> m\ge d/2$ and $m>1$. With this choice of $r$, the parameter $q_0$ defined in~\eqref{delta0} is given by
\begin{equation*}
	q_0 = \frac{2m+d(\rho-1)}{m(d+2)-2d} > 1.
\end{equation*}
We next define
\begin{equation*}
	\overline{u}_n:=(1+\tau)^{n} u_n,\quad \overline{v}_n:=(1+\tau)^{ n} v_n, \quad \overline{w}_n:=(1+\tau)^{ n} w_n, \qquad n\in\mathbb{N}.
\end{equation*}
It follows from~\eqref{ph04a} and~\eqref{ph04b} that, for $n\ge 1$, $\overline{v}_n$ and $\overline{w}_n$ solve
\begin{align*}
	\frac{\overline{v}_n-\overline{v}_{n-1}}{\tau}-\Delta \overline{v}_n+\overline{v}_n
	=\overline{w}_n+\overline{v}_{n-1} \;\;\text{ in }\;\; \mathbb{R}^d
	\eqntag
	\label{eqn;scheme-v}
\end{align*}
and
\begin{align*}
	\frac{\overline{w}_n-\overline{w}_{n-1}}{\tau}-\Delta \overline{w}_n+\overline{w}_n=(1+\tau)\overline{u}_{n-1}+\overline{w}_{n-1} \;\;\text{ in }\;\; \mathbb{R}^d,
	\eqntag
	\label{eqn;scheme-w}
\end{align*}
respectively. Since $r'=m$, we infer from~\eqref{zLs02} (with $z_n=\overline{v}_n$, $f_n=\overline{w}_n+\overline{v}_{n-1}$, $q=q_0$ and $s_0=m$) and~\eqref{eqn;scheme-v} that, for $N\ge 1$,
\begin{align*}
	& \left( \sum_{n=1}^N (1+\tau)^{q_0 n}\|v_n-\Delta v_n\|_{m}^{q_0} \right)^{1/q_0} =\, \left( \sum_{n=1}^N \|\overline{v}_n-\Delta \overline{v}_n\|_{m}^{q_0} \right)^{1/q_0} \\
	\le\, & C_0(q_0,m) \left[ \left( \sum_{n=1}^N \|\overline{w}_n + \overline{v}_{n-1}\|_{m}^{q_0} \right)^{1/q_0} + \tau^{-1/q_0} \big\|v_0 - \Delta v_0\big\|_{m} \right] \\
	\le\, & C(\rho) \left[ \left( \sum_{n=1}^N \|\overline{w}_n\|_m^{q_0} \right)^{1/q_0} + \left( \sum_{n=1}^N \|\overline{v}_{n-1}\|_{m}^{q_0} \right)^{1/q_0} + \tau^{-1/q_0} \big\|v_0 - \Delta v_0\big\|_{m} \right].
\end{align*}
We next deduce from~\eqref{zLs01} (with $z_n=\overline{w}_n$, $f_n = (1+\tau) \overline{u}_{n-1} + \overline{w}_{n-1}$, $q=q_0$ and $s_0=m$) and~\eqref{eqn;scheme-w} that, as $1+\tau\le 2$,
\begin{align*}
	& \left( \sum_{n=1}^N (1+\tau)^{q_0 n}\|v_n-\Delta v_n\|_{m}^{q_0} \right)^{1/q_0} \\
	\le\,& C(\rho) \left[ \left( \sum_{n=1}^N \|(1+\tau)\overline{u}_{n-1} + \overline{w}_{n-1}\|_{m}^{q_0} \right)^{1/q_0} + \tau^{-1/q_0} \|w_0\|_m \right] \\
	& + C(\rho) \left[ \left( \sum_{n=1}^N \|\overline{v}_{n-1}\|_{m}^{q_0} \right)^{1/q_0} + \tau^{-1/q_0} \left(\big\|v_0 - \Delta v_0\big\|_{m} \right) \right] \\ 
	\le\, & C(\rho) \left[ 2 \left( \sum_{n=1}^N \|\overline{u}_{n-1}\|_{m}^{q_0} \right)^{1/q_0} + \left( \sum_{n=1}^N \|\overline{w}_{n-1}\|_{m}^{q_0} \right)^{1/q_0} + \left( \sum_{n=1}^N \|\overline{v}_{n-1}\|_{m}^{q_0} \right)^{1/q_0}\right] \\
	& + C(\rho) \left[ \tau^{-1/q_0} \|w_0\|_m + \tau^{-1/q_0} \left(\big\|v_0 - \Delta v_0\big\|_{m} \right) \right] \\ 
	\le\,& C(\rho) \left[ \left(\sum_{n=1}^N (1+\tau)^{q_0(n-1)}  \|u_{n-1}\|_{m}^{q_0} \right)^{1/q_0} + \left(\sum_{n=1}^N (1+\tau)^{q_0(n-1)} \|w_{n-1}\|_{m}^{q_0} \right)^{1/q_0} \right] \\
	& + C(\rho) \left[ \left(\sum_{n=1}^N (1+\tau)^{q_0(n-1)} \|v_{n-1}\|_{m}^{q_0} \right)^{1/q_0} +\tau^{-1/q_0} \right].
	\eqntag
	\label{eqn;mr1}
\end{align*}	
At this point, it follows from~\eqref{ph23} and~\eqref{zLs03} (applied to $z_n=w_n$, $f_n=u_n-1$ and $z_n=v_n$, $f_n=w_n$ with $s_0=m$) that, for $n\ge 1$,
\begin{equation*}
	\|u_{n-1}\|_m \le B, \quad \|w_{n-1}\|_m \le \max\big\{ \|w_0\|_m , B \big\}
\end{equation*}
and
\begin{equation*}
	\|v_{n-1}\|_m \le \max\big\{ \|v_0\|_m , \max_{0\le j\le n-1} \|w_j\|_m \big\} \le \max\big\{ \|v_0\|_m, \|w_0\|_m, B \big\}.
\end{equation*}
Consequently, we infer from~\eqref{eqn;mr1} that
\begin{align*}
	\left( \sum_{n=1}^N (1+\tau)^{q_0 n}\|v_n-\Delta v_n\|_{m}^{q_0} \right)^{1/q_0} & \le C(\rho,B) \left[ \left( \sum_{n=1}^N (1+\tau)^{q_0(n-1)} \right)^{1/q_0} + \tau^{-1/q_0} \right] \\
	& = C(\rho,B) \left[ \left( \frac{(1+\tau)^{q_0 N}- 1}{(1+\tau)^{q_0}-1} \right)^{1/q_0} + \tau^{-1/q_0} \right] \\
	& \le C(\rho,B) \left[ (1+\tau)^N + 1 \right] \tau^{-1/q_0},
	\eqntag
	\label{eqn;mr2}
\end{align*}
after using the elementary inequality $(1+\tau)^{q_0} - 1\ge q_0 \tau > \tau$. Recalling that~\eqref{ph30} is satisfied, we now combine~\eqref{ph32} and~\eqref{eqn;mr2} to obtain
\begin{align*}
	& \frac{2m\rho M^{m-1}}{(m+\rho-1)^2}\sum_{n=1}^N (1+\tau)^{q_0 n} \big\|\nabla u_n^{(m+\rho-1)/2}\big\|_2^2 \\
	\le \,&\frac{1}{\tau(\rho-1)}\left[ \big[(1+\tau)^{q_0}-1\big] \sum_{n=1}^{N-1} (1+\tau)^{q_0 n} \|u_n\|_\rho^\rho + (1+\tau)^{q_0}\|u_{0}\|_{\rho}^{\rho} - (1+\tau)^{q_0 N} \|u_N\|_\rho^\rho \right]
	\\
	&+ \frac{C(\rho,B)}{\tau} \left[ (1+\tau)^N + 1 \right]^{q_0},
\end{align*}	
whence
\begin{align*}
(1+\tau)^{q_0 N} \|u_N\|_\rho^\rho &+ \frac{2m\rho(\rho-1) M^{m-1}}{(m+\rho-1)^2} \tau \sum_{n=1}^N (1+\tau)^{q_0 n}\big\|\nabla u_n^{(m+\rho-1)/2}\big\|_2^2 
\\
\le\,& \big[(1+\tau)^{q_0}-1\big] \sum_{n=1}^{N-1} (1+\tau)^{q_0 n} \|u_n\|_\rho^\rho + (1+\tau)^{q_0} \|u_{0}\|_{\rho}^{\rho} 
+ C(\rho,B) (1+\tau)^{q_0 N}.
\eqntag
\label{ph33} 
\end{align*}
We are left with exploiting the structure of~\eqref{ph33} to derive a bound on $\|u_N\|_\rho$ which does not depend on $N\ge 1$. For that purpose, since
\begin{equation*}
	m < \rho \le \frac{d(m+\rho-1)}{d-2}\quad(d\ge3),
	\quad m < \rho\quad(d\in\{1,2\})
\end{equation*}
we may apply \cite[Lemma~2.4]{Sug2006} (with $q_1=m$, $q_2=r=\rho$) to obtain
\begin{equation}
	\|u_n\|_\rho^\rho\le\,C(\rho) \|u_n\|_m^{\rho(1-\theta_1)}\|\nabla u_n^{(\rho+m-1)/2} \|_2^{2\delta_1}, \label{ph33b}
\end{equation}
where $\delta_1:=\theta_1\rho/(\rho+m-1)$ and
\begin{align*}
	\frac{2\theta_1}{\rho+m-1}:=\,& \left(\frac1m-\frac1{\rho}\right) \left(\frac{1}{ d}-\frac12+\frac{\rho+m-1}{2m}\right)^{-1}
	\\
	=\,&\frac{\rho-m}{m\rho} \frac{2md}{2m+d(\rho-1)}=\,\frac{2d(\rho-m)}{\rho[2m+d(\rho-1)]} >0.
\end{align*}
Observing that
\begin{equation*}
	1-\delta_1= 1-\frac{\theta_1 \rho}{\rho+m-1} = \frac{2m+d(m-1)}{2m+d(\rho-1)}>0,
\end{equation*}
it follows from~\eqref{ph23}, \eqref{ph33b} and Young's inequality that
\begin{align*}
	&\big[(1+\tau)^{q_0}-1\big] \sum_{n=1}^{N-1} (1+\tau)^{q_0 n} \|u_n\|_\rho^\rho
	\\
	\le\,& \big[(1+\tau)^{q_0}-1\big] C(\rho) B^{\rho(1-\theta_1)} \sum_{n=1}^{N} (1+\tau)^{q_0 n} \|\nabla u_n^{(\rho+m-1)/2}\|_2^{2\delta_1} \\
	\le\,& \frac{m\rho(\rho-1) M^{m-1}}{(m+\rho-1)^2} \tau \sum_{n=1}^{N} (1+\tau)^{q_0 n} \|\nabla u_n^{(\rho+m-1)/2}\|_2^2 \\
	& + C(\rho,B) \big[(1+\tau)^{q_0}-1\big]^{1/(1-\delta_1)} \tau^{-\delta_1/(1-\delta_1)} \sum_{n=1}^{N-1} (1+\tau)^{q_0 n} \\
	\le\,& \frac{m\rho(\rho-1) M^{m-1}}{(m+\rho-1)^2} \tau \sum_{n=1}^{N} (1+\tau)^{q_0 n} \|\nabla u_n^{(\rho+m-1)/2}\|_2^2 \\
	& + C(\rho,B) \left[ \frac{(1+\tau)^{q_0}-1}{\tau} \right]^{\delta_1/(1-\delta_1)} \big[ (1+\tau)^{q_0 N} - 1\big].
\end{align*}
Since $(1+\tau)^{q_0}-1 \le q_0 (1+\tau)^{q_0-1} \tau \le q_0 2^{q_0-1} \tau$, we conclude that 
\begin{align*}
	\big[(1+\tau)^{q_0}-1\big] \sum_{n=1}^{N-1} (1+\tau)^{q_0 n} \|u_n\|_\rho^\rho
	\le\,& \frac{m\rho(\rho-1) M^{m-1}}{(m+\rho-1)^2} \tau \sum_{n=1}^{N} (1+\tau)^{q_0 n} \|\nabla u_n^{(\rho+m-1)/2}\|_2^2 \\
	& + C(\rho,B) (1+\tau)^{q_0 N}.
	\eqntag
	\label{eqn;unrho}
\end{align*}
Combining~\eqref{ph33} and~\eqref{eqn;unrho} leads us to
\begin{align*}
	(1+\tau)^{q_0 N} \|u_N\|_\rho^\rho &+ \frac{m\rho(\rho-1) M^{m-1}}{(m+\rho-1)^2} \tau \sum_{n=1}^N (1+\tau)^{q_0 n}\big\|\nabla u_n^{(m+\rho-1)/2}\big\|_2^2 
	\\
	\le\,& (1+\tau)^{q_0} \|u_{0}\|_{\rho}^{\rho}  + C(\rho,B) (1+\tau)^{q_0 N},
\end{align*}
whence
\begin{align*}
	\|u_N\|_\rho^\rho +\, & \frac{m\rho(\rho-1) M^{m-1}}{(m+\rho-1)^2} \tau \sum_{n=1}^N (1+\tau)^{-q_0(N-n)} \big\|\nabla u_n^{(m+\rho-1)/2}\big\|_2^2 \\
	& \le\, (1+\tau)^{-q_0(N-1)} \|u_{0}\|_{\rho}^{\rho} +C(\rho,B) \le\, C(\rho,B),
\end{align*}
and we have proved~\eqref{ph24} in that case.

\medskip 

\noindent\textbf{Step~4.} We finally turn to the case $m\in (1,d/2)$ when $d\in\{3,4\}$ and $m\in (m_*,d/2)$ when $d\ge 5$. In this step only, we assume additionally that $\rho> \max\{m,(d-2)/2\}$ and pick $s\in (1,d/2)$ such that
\begin{equation}
	\max\left\{ \frac{d}{4}, m \right\} < s < \frac{d}{2} \le \frac{d(m+\rho-1)}{d-2}, \label{ph34}
\end{equation}
which is possible in view of the constraint $\rho>(d-2)/2$. We then choose 
\begin{equation}
	1 < r = \frac{ds}{s(d+2)-d} \le r'=\frac{ds}{d-2s}, \label{rs}
\end{equation}
the ordering $r'>r$ being due to $s>m>m_*$ when $d\ge 5$ and to $s>m>1$ when $d\in\{3,4\}$. With this choice of $r$, the parameter $q_0$ defined in~\eqref{delta0} is given by
\begin{equation*}
	q_0 = \frac{s[2m+d(\rho-1)]}{s[m(d+4)-d]-md}>1.
\end{equation*}
As in the derivation of~\eqref{eqn;mr1}, we first infer from~\eqref{eqn;scheme-v} and~\eqref{zLs02} (with $z_n=\overline{v}_n$, $f_n=\overline{w}_n + \overline{v}_{n-1}$, $q=q_0$ and $s_0=s$) that, for $N\ge 1$,
\begin{align*}
	& \left( \sum_{n=1}^N (1+\tau)^{q_0 n} \|v_n-\Delta v_n\|_{r'}^{q_0} \right)^{1/q_0} \\
	&\le\, C_0(q_0,r') \left[ \left( \sum_{n=1}^N (1+\tau)^{q_0 n} \|w_n\|_{r'}^{q_0} \right)^{1/q_0} + \left( \sum_{n=1}^N (1+\tau)^{q_0 n} \|v_{n-1}\|_{r'}^{q_0} \right)^{1/q_0} \right] \\
	& \quad + C_0(q_0,r') \tau^{-1/q_0} \|v_0-\Delta v_0\|_{r'}.
\end{align*}
Owing to the continuous embedding of $W^{2,s}(\mathbb{R}^d)$ in $L^{r'}(\mathbb{R}^d)$, there is a positive constant $C(s)$ depending only on $d$ and $s$ such that
\begin{equation*}
	\|w_n\|_{r'} \le C(s) \|w_n-\Delta w_n\|_{s}, \qquad n\ge 1,
\end{equation*}
and it follows from~\eqref{eqn;scheme-w} and~\eqref{zLs02} (with $z_n=\overline{w}_n$, $f_n= (1+\tau) \overline{u}_{n-1} + \overline{w}_{n-1}$, $q=q_0$ and $s_0=s$) that, for $N\ge 1$,
\begin{align*}
	& \left( \sum_{n=1}^N (1+\tau)^{q_0 n} \|w_n\|_{r'}^{q_0} \right)^{1/q_0} \le\, C(s) \left( \sum_{n=1}^N (1+\tau)^{q_0 n} \|w_n-\Delta w_n\|_{s}^{q_0} \right)^{1/q_0} \\
	&\le\, C(s) C_0(q_0,s) \left[ \left( \sum_{n=1}^N (1+\tau)^{q_0 n} \|u_{n-1}\|_{s}^{q_0} \right)^{1/q_0} + \left( \sum_{n=1}^N (1+\tau)^{q_0 (n-1)} \|w_{n-1}\|_{s}^{q_0} \right)^{1/q_0} \right] \\
	& \quad + C(s) C_0(q_0,s) \tau^{-1/q_0} \|w_0-\Delta w_0\|_{s}.
\end{align*}
Combining the previous two inequalities, we obtain 
\begin{align*}
	& \left( \sum_{n=1}^N (1+\tau)^{q_0 n} \|v_n-\Delta v_n\|_{r'}^{q_0} \right)^{1/q_0} \\
	&\le\, C(\rho,s) \left[ \left( \sum_{n=1}^N (1+\tau)^{q_0 (n-1)} \|u_{n-1}\|_{s}^{q_0} \right)^{1/q_0} + \left( \sum_{n=1}^N (1+\tau)^{q_0 (n-1)} \|w_{n-1}\|_{s}^{q_0} \right)^{1/q_0} \right] \\
	& \quad + C(\rho,s) \left[ \left( \sum_{n=1}^N (1+\tau)^{q_0 (n-1)} \|v_{n-1}\|_{r'}^{q_0} \right)^{1/q_0} + \tau^{-1/q_0} \right].
\eqntag
\label{eqn;Ls}
\end{align*}
Let us first estimate the sum involving $(w_{n-1})_{1\le n\le N}$ in~\eqref{eqn;Ls}. Since $m<s<r'=ds/(d-2s)$ and $W^{2,s}(\mathbb{R}^d)$ embeds continuously in $L^{r'}(\mathbb{R}^d)$, we infer from H\"older's and Sobolev's inequalities that, setting $\theta_2 := [r'(s-m)]/[s(r'-m)]\in (0,1)$, 
\begin{align*}
	\sum_{n=1}^N (1+\tau)^{q_0 (n-1)}\|w_{n-1}\|_s^{q_0}\le\,&\sum_{n=1}^N (1+\tau)^{q_0 (n-1)}\|w_{n-1}\|_m^{q_0(1-\theta_2)} \| w_{n-1}\|_{r'}^{q_0 \theta_2}
	\\
	\le\,& C(\rho,s) \sum_{n=1}^{N} (1+\tau)^{q_0 (n-1)} \|w_{n-1}\|_{m}^{q_0(1-\theta_2)} \|w_{n-1}-\Delta w_{n-1}\|_{s}^{q\theta_2}. 
	\eqntag
	\label{eqn;Ls-w0}
\end{align*}
On the one hand, we deduce from~\eqref{ph04b}, \eqref{ph23} and~\eqref{zLs03} (with $z_n=w_n$, $f_n=u_{n-1}$ and $s_0=m$) that
\begin{equation}
	\|w_{n-1}\|_m \le \max\{\|w_0\|_m,B\}, \qquad 1\le n \le N. \label{ph50}
\end{equation}
On the other hand, it follows from~\eqref{eqn;scheme-w} and~\eqref{zLs02} (with $z_n=\overline{w}_n$, $f_n=(1+\tau) \overline{u}_{n-1} + \overline{w}_{n-1}$, $q=q_0$ and $s_0=s$) that
\begin{align*}
	& \left( \sum_{n=1}^N (1+\tau)^{q_0 (n-1)} \|w_{n-1}-\Delta w_{n-1}\|_{s}^{q_0} \right)^{1/q_0} \\
	\le\, & \|w_0-\Delta w_0\|_s + \left( \sum_{n=1}^{N-1}(1+\tau)^{q_0 n} \|w_{n}-\Delta w_{n}\|_{s}^{q_0} \right)^{1/q_0} \\
	\le\, & \|w_0-\Delta w_0\|_s + C_0(q_0,s) \left[ \left( \sum_{n=1}^N \big\|(1+\tau) \overline{u}_{n-1} + \overline{w}_{n-1} \big\|_{s}^{q_0} \right)^{1/q_0} + \tau^{-1/q_0} \|w_0-\Delta w_0\|_s \right] \\
	\le\, & C(\rho,s) \left[ \left( \sum_{n=1}^N (1+\tau)^{q_0 (n-1)} \|u_{n-1}\|_{s}^{q_0} \right)^{1/q_0} + \left( \sum_{n=1}^N (1+\tau)^{q_0 (n-1)} \|w_{n-1}\|_{s}^{q_0} \right)^{1/q_0} + \tau^{-1/q_0} \right].
	\eqntag
	\label{ph51}
\end{align*}
Combining~\eqref{eqn;Ls-w0}, \eqref{ph50} and Young's inequality gives, for $\varepsilon>0$ yet to be determined, 
\begin{align*}
	\sum_{n=1}^N (1+\tau)^{q_0 (n-1)}\|w_{n-1}\|_s^{q_0}\le\,& C(\rho,s,B) \sum_{n=1}^N (1+\tau)^{q_0 (n-1)} \| w_{n-1} - \Delta w_{n-1}\|_{s}^{q_0 \theta_2} \\
	\le\,& \varepsilon^{q_0} \sum_{n=1}^N (1+\tau)^{q_0 (n-1)} \| w_{n-1} - \Delta w_{n-1}\|_{s}^{q_0} \\
	& + C(\rho,s,B,\varepsilon) \sum_{n=1}^N (1+\tau)^{q_0 (n-1)},
\end{align*}
whence, using now~\eqref{ph51},
\begin{align*}
	\left( \sum_{n=1}^N (1+\tau)^{q_0 (n-1)}\|w_{n-1}\|_s^{q_0} \right)^{1/q_0} \le\,& \varepsilon C(\rho,s) \left( \sum_{n=1}^N (1+\tau)^{q_0 (n-1)} \|u_{n-1}\|_{s}^{q_0} \right)^{1/q_0} \\
	& + \varepsilon C(\rho,s) \left( \sum_{n=1}^N (1+\tau)^{q_0 (n-1)} \|w_{n-1}\|_{s}^{q_0} \right)^{1/q_0} \\
	& + \varepsilon C(\rho,s) \tau^{-1/q_0} + C(\rho,s,B,\varepsilon) \left( \frac{(1+\tau)^{q_0 N}-1}{(1+\tau)^{q_0}-1} \right)^{1/q_0} \\
	\le\,& \varepsilon C(\rho,s) \left( \sum_{n=1}^N (1+\tau)^{q_0 (n-1)} \|u_{n-1}\|_{s}^{q_0} \right)^{1/q_0} \\
	& + \varepsilon C(\rho,s) \left( \sum_{n=1}^N (1+\tau)^{q_0 (n-1)} \|w_{n-1}\|_{s}^{q_0} \right)^{1/q_0} \\
	& + C(\rho,s,B,\varepsilon) (1+\tau)^N \tau^{-1/q_0}.
\end{align*}
Choosing $\varepsilon=1/[2C(\rho,s)]$ in the above inequality yields
\begin{equation}
\begin{split}
	\left( \sum_{n=1}^N (1+\tau)^{q_0 (n-1)}\|w_{n-1}\|_s^{q_0} \right)^{1/q_0} \le\,& \left( \sum_{n=1}^N (1+\tau)^{q_0 (n-1)} \|u_{n-1}\|_{s}^{q_0} \right)^{1/q_0} \\
	& + C(\rho,s,B) (1+\tau)^N \tau^{-1/q_0}. 
\end{split} \label{eqn;Ls-w}
\end{equation}

We now turn to the sum involving $(v_{n-1})_{1\le n\le N}$ in~\eqref{eqn;Ls} and first notice that the constraints~\eqref{ph34} on $s$ imply that $s>2d/(d+4)$ and thus $r'>2$. We split the analysis into two cases and first consider the case $s>2$. Then $s\in (2,r')$ and it follows from the continuous embedding of $W^{2,s}(\mathbb{R}^d)$ in $L^{r'}(\mathbb{R}^d)$, Corollary~\ref{cords04} and H\"older's and Young's inequalities that, for $\varepsilon>0$ yet to be determined and 
\begin{equation*}
	\theta_3 := \frac{r'(s-2)}{s(r'-2)} = \frac{d(s-2)}{s(d+4)-2d} \in (0,1),
\end{equation*}
\begin{align*}
	& \left( \sum_{n=1}^N (1+\tau)^{q_0(n-1)} \|v_{n-1}\|_{r'}^{q_0} \right)^{1/q_0} \\
	\le\; & C(s) \left( \sum_{n=1}^N (1+\tau)^{q_0(n-1)} \|v_{n-1}-\Delta v_{n-1}\|_{s}^{q_0} \right)^{1/q_0} \\
	\le\; & C(s) \left( \sum_{n=1}^N (1+\tau)^{q_0(n-1)} \|v_{n-1}-\Delta v_{n-1}\|_{2}^{q_0(1-\theta_3)} \|v_{n-1}-\Delta v_{n-1}\|_{r'}^{q_0\theta_3} \right)^{1/q_0} \\
	\le\; & C(\rho,s) \left( \sum_{n=1}^N (1+\tau)^{q_0(n-1)} \|v_{n-1}-\Delta v_{n-1}\|_{r'}^{q_0\theta_3} \right)^{1/q_0} \\
	\le\; & C(\rho,s) \left[ \|v_0-\Delta v_0\|_s + \left( \sum_{n=1}^{N-1} (1+\tau)^{q_0 n} \|v_{n}-\Delta v_{n}\|_{r'}^{q_0\theta_3} \right)^{1/q_0} \right] \\
	\le\; & C(\rho,s) \left[ 1 + \left( \varepsilon^{q_0} \sum_{n=1}^{N-1} (1+\tau)^{q_0 n} \|v_{n}-\Delta v_{n}\|_{r'}^{q_0} + C(\rho,s,\varepsilon) \sum_{n=1}^{N-1} (1+\tau)^{q_0 n} \right)^{1/q_0} \right] \\
	\le\; & \varepsilon C(\rho,s) \left( \sum_{n=1}^{N-1} (1+\tau)^{q_0 n} \|v_{n}-\Delta v_{n}\|_{r'}^{q_0} \right)^{1/q_0} + C(\rho,s,\varepsilon) \left[ 1 + \left( \frac{(1+\tau)^{q_0 N} - 1}{(1+\tau)^{q_0}-1} \right)^{1/q_0} \right] \\
	\le\; & \varepsilon C(\rho,s) \left( \sum_{n=1}^{N-1} (1+\tau)^{q_0 n} \|v_{n}-\Delta v_{n}\|_{r'}^{q_0} \right)^{1/q_0} + C(\rho,s,\varepsilon) (1+\tau)^N \tau^{-1/q_0}.
	\eqntag
	\label{eqn;Ls-v1}
\end{align*}
Next, if $s<2$, then $r'\in (2,2d/(d-4))$ for $d\ge 5$ and $r'\in (2,\infty)$ for $d\in\{3,4\}$, from which we deduce that $W^{2,2}(\mathbb{R}^d)$ is continuously embedded in $L^{r'}(\mathbb{R}^d)$. An immediate consequence of Corollary~\ref{cords04} is then that
\begin{align*}
	\left( \sum_{n=1}^N (1+\tau)^{q_0(n-1)} \|v_{n-1}\|_{r'}^{q_0} \right)^{1/q_0} \le\; & C(s) \left( \sum_{n=1}^N (1+\tau)^{q_0(n-1)} \|v_{n-1}-\Delta v_{n-1}\|_{2}^{q_0} \right)^{1/q_0} \\
	\le\; & C(\rho,s) \left( \sum_{n=1}^N (1+\tau)^{q_0(n-1)} \right)^{1/q_0} \\
	\le\; & C(\rho,s) (1+\tau)^N \tau^{-1/q_0}.
	\eqntag
	\label{eqn;Ls-v2}
\end{align*}
In view of~\eqref{eqn;Ls-v1} and~\eqref{eqn;Ls-v2}, we have shown that, for $\varepsilon>0$, 
\begin{align*}
	\left( \sum_{n=1}^N (1+\tau)^{q_0(n-1)} \|v_{n-1}\|_{r'}^{q_0} \right)^{1/q_0} \le\; & \varepsilon C(\rho,s) \left( \sum_{n=1}^{N-1} (1+\tau)^{q_0 n} \|v_{n}-\Delta v_{n}\|_{r'}^{q_0} \right)^{1/q_0} \\
	& + C(\rho,s,\varepsilon) (1+\tau)^N \tau^{-1/q_0}.
	\eqntag
	\label{eqn;Ls-v}
\end{align*}
Now, \eqref{eqn;Ls} is estimated by combining~\eqref{eqn;Ls-w} and~\eqref{eqn;Ls-v} as 
\begin{align*}
	\left( \sum_{n=1}^N (1+\tau)^{q_0 n} \|v_n-\Delta v_n\|_{r'}^{q_0} \right)^{1/q_0} 
	\le\; & \varepsilon C(\rho,s) \left( \sum_{n=1}^{N-1} (1+\tau)^{q_0n} \|v_n-\Delta v_n\|_{r'}^{q_0} \right)^{1/q_0} \\
	& + C(\rho,s) \left( \sum_{n=1}^N (1+\tau)^{q_0 n} \|u_{n-1}\|_{s}^{q_0} \right) \\
	& + C(\rho,s,B,\varepsilon) (1+\tau)^N \tau^{-1/q_0},
\end{align*}
and the choice $\varepsilon=1/[2C(\rho,s)]$ leads us to the estimate
\begin{align*}
	\left( \sum_{n=1}^N (1+\tau)^{q_0 n} \|v_n-\Delta v_n\|_{r'}^{q_0} \right)^{1/q_0} 
	\le\; & C(\rho,s) \left( \sum_{n=1}^N (1+\tau)^{q_0 n} \|u_{n-1}\|_{s}^{q_0} \right) \\
	& + C(\rho,s,B) (1+\tau)^N \tau^{-1/q_0}.
	\eqntag
	\label{eqn;Ls-u}
\end{align*}
We now aim at combining~\eqref{eqn;Ls-u} with~\eqref{ph32} to proceed further. Using the latter however requires to check that the choice~\eqref{rs} of $r$, along with~\eqref{ph34}, ensures the validity of~\eqref{ph30}. In terms of $s$, we need to guarantee the positivity of 
\begin{equation*}
	\frac{ds\rho}{s(d+2)-d - m} \;\;\text{ and }\;\; \frac{d}{d-2}(m+\rho-1) - \frac{\rho d s}{s(d+2)-d - m}.
\end{equation*}
To this end, we notice that the constraints $\rho>m$ and $s<d/2$ imply that
\begin{equation*}
	\frac{ds\rho}{s(d+2)-d} - m = \frac{sd(\rho-m)+m(d-2s)}{s(d+2)-d} \ge \frac{m(d-2s)}{s(d+2)-d}>0,
\end{equation*}
while it follows from the constraint $s>d/4$ that
\begin{align*}
	\frac{d(m+\rho-1)}{d-2} - \frac{ds\rho}{s(d+2)-d} & = \frac{d}{d-2} \frac{(m+\rho-1)[s(d+2)-d] -s\rho(d-2)}{s(d+2)-d} \\
	& = \frac{d}{d-2} \frac{(m-1)[s(d+2)-d] + \rho (4s-d)}{s(d+2)-d} > 0.
\end{align*}
Consequently, the condition~\eqref{ph30} is verified and, plugging~\eqref{eqn;Ls-u} in~\eqref{ph32}, we obtain
\begin{align*}	
	& \frac{2m\rho M^{m-1}}{(m+\rho-1)^2}\sum_{n=1}^N (1+\tau)^{q_0 n} \big\|\nabla u_n^{(m+\rho-1)/2}\big\|_2^2 \\
	\le \,& \frac{1}{\tau(\rho-1)} \left[ \big[(1+\tau)^{q_0}-1\big] \sum_{n=1}^{N-1} (1+\tau)^{q_0 n} \|u_{n}\|_\rho^\rho + (1+\tau)^{q_0} \|u_{0}\|_{\rho}^{\rho} - (1+\tau)^{q_0 N} \|u_N\|_\rho^\rho \right] \\
	&+ C(\rho,s,B) \left( \sum_{n=1}^{N-1} (1+\tau)^{q_0 n} \|u_{n}\|_{s}^{q_0} + (1+\tau)^{q_0 N} \tau^{-1}  \right).
	\eqntag
	\label{ph35}
\end{align*}
To estimate the right-hand side of~\eqref{ph35}, we use once more \cite[Lemma~2.4]{Sug2006} (with $q_1=m$, $q_2=s$ and $r=\rho$), which can be applied owing to~\eqref{ph34}, and obtain
\begin{equation}
	\|u_n\|_s \le C(\rho,s) \big\|\nabla u_n^{(m+\rho-1)/2}\big\|_2^{2\theta_4/(m+\rho-1)} \|u_n\|_m^{1-\theta_4}, \label{ph36}
\end{equation}
with $\theta_4\in (0,1)$ given by
\begin{align*}
	\frac{2\theta_4}{m+\rho-1} & := \left( \frac{1}{m} - \frac{1}{s} \right) \left( \frac{1}{d} - \frac{1}{2} + \frac{m+\rho-1}{2m} \right)^{-1} \\
	& = \frac{s-m}{ms} \frac{2md}{2m+d(\rho-1)} = \frac{2d(s-m)}{s[2m+d(\rho-1)]}.
\end{align*}
It then follows from~\eqref{ph23}, \eqref{ph35} and~\eqref{ph36} that
\begin{align*}
	& \frac{2m\rho M^{m-1}}{(m+\rho-1)^2} \sum_{n=1}^N (1+\tau)^{q_0 n} \big\|\nabla u_n^{(m+\rho-1)/2}\big\|_2^2\\
	\le\; & \frac{1}{\tau(\rho-1)}\left[ \big[(1+\tau)^{q_0}-1\big] \sum_{n=1}^{N-1} (1+\tau)^{q_0 n} \|u_{n}\|_\rho^\rho + (1+\tau)^{q_0} \|u_{0}\|_{\rho}^{\rho} - (1+\tau)^{q_0 N} \|u_N\|_\rho^\rho \right]\\
	&  + C(\rho,s,B) \left[ \sum_{n=1}^{N-1} (1+\tau)^{q_0 n} \big\|\nabla u_n^{(m+\rho-1)/2}\big\|_2^{2\theta_4 q_0/(m+\rho-1)} +(1+\tau)^{q_0 N} \tau^{-1} \right].
\end{align*}
By~\eqref{delta0} (with $r=ds/[s(d+2)-d]$) and~\eqref{rs}, 
\begin{align*}
	1 - \frac{q_0\theta_4}{m+\rho-1} =\; & 1 - \frac{d(s-m)}{s[m(d+4)-d]-md} = \frac{s(d+4)(m-m_*)}{(d+4)ms - d(m+s)}>0,
\end{align*}
so that another application of Young's inequality gives
\begin{align*}
	&\frac{2m\rho M^{m-1}}{(m+\rho-1)^2} \sum_{n=1}^N (1+\tau)^{q_0 n} \big\|\nabla u_n^{(m+\rho-1)/2}\big\|_2^2\\
	\le\; &\frac{1}{\tau(\rho-1)}\left[ \big[(1+\tau)^{q_0}-1\big] \sum_{n=1}^{N-1} (1+\tau)^{q_0 n} \|u_{n}\|_\rho^\rho + (1+\tau)^{q_0} \|u_{0}\|_{\rho}^{\rho} - (1+\tau)^{q_0 N} \|u_N\|_\rho^\rho \right]\\
	& + \frac{m\rho M^{m-1}}{2(m+\rho-1)^2} \sum_{n=1}^{N-1} (1+\tau)^{q_0 n} \big\|\nabla u_n^{(m+\rho-1)/2}\big\|_2^{2} \\
	& + C(\rho,s,B) \left( \sum_{n=1}^{N-1} (1+\tau)^{q_0 n} + (1+\tau)^{q_0 N} \tau^{-1} \right).
\end{align*}
Hence,
\begin{align*}
	&\frac{3m\rho M^{m-1}}{2(m+\rho-1)^2} \sum_{n=1}^N (1+\tau)^{q_0 n} \big\|\nabla u_n^{(m+\rho-1)/2}\big\|_2^2\\
	\le\; &\frac{1}{\tau(\rho-1)}\left[ \big[(1+\tau)^{q_0}-1\big] \sum_{n=1}^{N-1} (1+\tau)^{q_0 n} \|u_{n}\|_\rho^\rho + (1+\tau)^{q_0} \|u_{0}\|_{\rho}^{\rho} - (1+\tau)^{q_0 N} \|u_N\|_\rho^\rho \right]\\
	& + C(\rho,s,B) (1+\tau)^{q_0 N} \tau^{-1}.
\end{align*}
Consequently, recalling~\eqref{eqn;unrho},
\begin{align*}
	(1+\tau)^{q_0 N} \|u_N\|_\rho^\rho +\; & \frac{3m\rho(\rho-1) M^{m-1}}{2(m+\rho-1)^2} \tau \sum_{n=1}^N (1+\tau)^{q_0 n} \big\|\nabla u_n^{(m+\rho-1)/2}\big\|_2^2 \\
	\le\;& (1+\tau)^{q_0} \|u_0\|_\rho^\rho + \big[(1+\tau)^{q_0}-1\big]n\sum_{n=1}^{N-1} (1+\tau)^{q_0 n} \|u_{n}\|_\rho^\rho \\
	& + C(\rho,s,B) (1+\tau)^{q_0 N} \\
	\le\; & (1+\tau)^{q_0} \|u_0\|_\rho^\rho +\frac{m\rho(\rho-1) M^{m-1}}{(m+\rho-1)^2} \tau \sum_{n=1}^N (1+\tau)^{q_0 n} \big\|\nabla u_n^{(m+\rho-1)/2}\big\|_2^2 \\
	& + C(\rho,s,B) (1+\tau)^{q_0 N},
\end{align*}
from which~\eqref{ph24} follows. We have thus completed the proof of Lemma~\ref{lemds04}.
\end{proof}

\subsection{$L^\infty$-estimate on $(u_n)_{n\ge 1}$}\label{sec.2.5}

Having derived $L^\rho$-bounds on $(u_n)_{n\ge 1}$ for $\rho\in (1,\infty)$ large enough, we improve the regularity of $(w_n)_{n\ge 1}$ and $(v_n)_{n\ge 1}$ with the help of discrete analogues of classical smoothing effects of the linear heat equation, which we establish in Lemma~\ref{lemA3}. Specifically, we first show that the $L^{\rho}$-estimate on $(u_n)_{n\ge 1}$ for $\rho>d$ provided by Lemma~\ref{lemds04} implies the boundedness of $(w_n)_{n\ge 1}$ in $W^{1,\infty}(\mathbb{R}^d)$, which in turn entails that of $(v_n-\Delta v_n)_{n\ge 1}$ in $L^\infty(\mathbb{R}^d)$. 

\begin{lemma}\label{lemds05}
Let $d\ge 1$ and $m>1$ satisfying additionally $m>m_*$ when $d\ge 5$ and assume that there is $B>0$ such that~\eqref{ph23} holds true. Then there is $C(B)>0$ depending on $d$, $m$ and $B$ such that
\begin{equation}
	\|w_n\|_{W^{1,\infty}} + \|v_n-\Delta v_n\|_{\infty} \le C(B), \qquad n\ge 1. \label{ph42}
\end{equation}
\end{lemma}

\begin{proof}
Fix $p\in (d,\infty)$. By Lemma~\ref{lemds04}, $\|u_{n-1}\|_p \le \max\{\|u_0\|_p,B_p\}$ for $n\ge 1$. Owing to this property and~\eqref{ph04b}, we are in a position to apply the first statement in Lemma~\ref{lemA3} (with $z_n=w_n$ and $f_n = u_{n-1}$) to obtain
\begin{equation}
	\|w_n\|_{W^{1,\infty}} \le C_1(p) \left( \|w_0\|_{W^{1,\infty}} +  \max\{\|u_0\|_p,B_p\} \right), \qquad n\ge 1. \label{ph60}
\end{equation}
In view of~\eqref{ph04a} and~\eqref{ph60}, we may apply the second statement in Lemma~\ref{lemA3} (with $z_n=v_n$ and $f_n = w_n$) and conclude that
\begin{equation*}
	\|v_n-\Delta v_n\|_\infty \le C_2 \left[ \|z_0-\Delta z_0\|_\infty + C_1(p) \left( \|w_0\|_{W^{1,\infty}} +  \max\{\|u_0\|_p,B_p\} \right) \right], \qquad n\ge 1,
\end{equation*}
and the proof is complete.
\end{proof}

We are now in a position to derive a uniform bound on $(u_n)_{n\ge 1}$ by a Moser iteration technique \cite{Ali1979}.
	
\begin{lemma}\label{lemds06}
Let $d\ge 1$ and $m>1$ satisfying additionally $m>m_*$ when $d\ge 5$ and assume that there is $B>0$ such that~\eqref{ph23} holds true. Then, there is $C(B)>0$ depending on $d$, $m$ and $B$ such that
\begin{equation*}
	\|u_n\|_{\infty} \le C(B), \qquad n\ge1. 
\end{equation*}
\end{lemma}

\begin{proof}
Let $\rho> \max\{m,d\}$. The starting point of the proof is the inequality~\eqref{ph29} and the estimate~\eqref{ph42}, from which we deduce that, for $n\ge 1$,
\begin{align*}
	\frac{4m\rho(\rho-1)M^{m-1}}{(m+\rho-1)^2} \big\|\nabla u_n^{(m+\rho-1)/2}\big\|_2^2 & \le \frac{\|u_{n-1}\|_{\rho}^{\rho} - \|u_n\|_\rho^\rho}{\tau} + M(\rho-1) \|v_n - \Delta v_n\|_\infty \|u_n\|_\rho^\rho \\
	& \le \frac{\|u_{n-1}\|_{\rho}^{\rho} - \|u_n\|_\rho^\rho}{\tau} - \|u_n\|_\rho^\rho + \rho C(B) \|u_n\|_\rho^\rho.
\end{align*}
Owing to the elementary inequalities 
\begin{equation*}
	\frac{\rho}{m+\rho-1} \ge \frac{m}{2m-1} \ge \frac{1}{2} \;\;\text{ and }\;\; \frac{\rho-1}{m+\rho-1} \ge \frac{m-1}{2m-1} \ge \frac{1}{2m}
\end{equation*}
 for $\rho\ge m$, we realize that $4\rho(\rho-1)/(m+\rho-1)^2 \ge (m-1)/m$ for $\rho\ge m$ and we further obtain
\begin{equation}
	(m-1)M^{m-1} \big\|\nabla u_n^{(m+\rho-1)/2}\big\|_2^2 \le \frac{\|u_{n-1}\|_{\rho}^{\rho} - \|u_n\|_\rho^\rho}{\tau} - \|u_n\|_\rho^\rho + \rho C_1(B)\|u_n\|_\rho^\rho. \label{ph44}
\end{equation}

We next split the analysis according to the space dimension and begin with higher space dimensions $d\ge 3$, for which Sobolev's inequality is available.

\medskip

\noindent $\triangleright$ $d\ge 3$: By H\"older's and Sobolev's inequalities,
\begin{align*}
	\|u_n\|_\rho^\rho & = \left\| u_n^{(m+\rho-1)/2} \right\|_{2\rho/(m+\rho-1)}^{2\rho/(m+\rho-1)} \le \left\| u_n^{(m+\rho-1)/2} \right\|_{2d/(d-2)}^{2\rho\theta/(m+\rho-1)} \left\| u_n^{(m+\rho-1)/2} \right\|_{\rho/(m+\rho-1)}^{2\rho(1-\theta)/(m+\rho-1)} \\
	& \le \left( S_d^2 \left\| \nabla u_n^{(m+\rho-1)/2} \right\|_{2}^2 \right)^{\rho\theta/(m+\rho-1)} \|u_n\|_{\rho/2}^{\rho(1-\theta)} 
\end{align*}
with
\begin{equation*}
	\theta := \frac{d(m+\rho-1)}{\rho(d+2)+2d(m-1)}>\frac{1}{2}, \qquad 1-\theta = \frac{2\rho+d(m-1)}{\rho(d+2)+2d(m-1)}.
\end{equation*}
Observing that
\begin{equation*}
	1 - \frac{\rho\theta}{m+\rho-1} = \frac{2\rho +2d(m-1)}{\rho(d+2)+2d(m-1)}>0,
\end{equation*}
we infer from Young's inequality that
\begin{align*}
	\rho C_1(B) \|u_n\|_\rho^\rho & \le \frac{\rho\theta}{m+\rho-1} (m-1)M^{m-1} \left\| \nabla u_n^{(m+\rho-1)/2} \right\|_{2}^2 \\
	& \qquad + \frac{\rho(1-\theta)+m-1}{m+\rho-1} \left( \frac{S_d^2}{(m-1)M^{m-1}} \right)^{\rho\theta/[\rho(1-\theta)+m-1]} \\
	& \qquad\qquad \times \left( \rho C_{2}(B) \|u_n\|_{\rho/2}^{\rho(1-\theta)} \right)^{(m+\rho-1)/[\rho(1-\theta)+m-1]}.
\end{align*}
At this point, we notice that
\begin{align*}
	\frac{\rho\theta}{m+\rho-1} & = \frac{\rho d}{\rho(d+2)+2d(m-1)} \le \frac{d}{d+2}, \\
	\rho(1-\theta) + m - 1 & = \frac{2(m+\rho-1)[\rho+d(m-1)]}{\rho(d+2)+2d(m-1)}, \\
	\frac{\rho\theta}{\rho(1-\theta) + m - 1} & = \frac{\rho d}{2\rho+2d(m-1)} \le \frac{d}{2}, \\
	\frac{m+\rho-1}{\rho(1-\theta) + m - 1} & = \frac{\rho(d+2)+2d(m-1)}{2\rho+2d(m-1)} \le \frac{d+2}{2}, \\
	\frac{\rho(1-\theta)(m+\rho-1)}{\rho(1-\theta) + m - 1} & = \frac{\rho}{2} \frac{2\rho+d(m-1)}{\rho+d(m-1)} \le \rho, 
\end{align*}
and end up with
\begin{align*}
	\rho C_{1}(B) \|u_n\|_\rho^\rho & \le \frac{d(m-1)M^{m-1}}{d+2} \left\| \nabla u_n^{(m+\rho-1)/2} \right\|_{2}^2 \\
	& \qquad + C_{2}(B) \rho^{(d+2)/2} \left( \max\{1,\|u_n\|_{\rho/2}^{\rho/2}\} \right)^2.
\eqntag
\label{eqn;un_Lrho}
\end{align*}
Combining~\eqref{ph44} and~\eqref{eqn;un_Lrho} gives
\begin{align*}
	\frac{2(m-1)M^{m-1}}{d+2} \big\|\nabla u_n^{(m+\rho-1)/2}\big\|_2^2 & \le \frac{\|u_{n-1}\|_{\rho}^{\rho} - \|u_n\|_\rho^\rho}{\tau} - \|u_n\|_\rho^\rho \\
	& + C_{2}(B) \rho^{(d+2)/2} \left( \max\{1,\|u_n\|_{\rho/2}^{\rho/2}\} \right)^2,
\end{align*}
whence
\begin{equation*}
	\|u_n\|_\rho^\rho \le \frac{1}{1+\tau} \|u_{n-1}\|_\rho^\rho + \frac{\tau}{1+\tau} C(B) \rho^{(d+2)/2} \left( \max\{1,\|u_n\|_{\rho/2}^{\rho/2}\} \right)^2, \qquad n\ge 1.
\end{equation*}
Setting $U_\rho := \sup_{n\ge 0} \left( \max\{1,\|u_n\|_{\rho}^{\rho}\} \right)$, which is finite according to Lemma~\ref{lemds04}, we iterate the above formula to find, for $n\ge 1$,
\begin{align*}
	\|u_n\|_\rho^\rho \le\; & \frac{1}{(1+\tau)^n} \|u_0\|_\rho^\rho + \tau C(B) \rho^{(d+2)/2} \sum_{j=1}^n \frac{1}{(1+\tau)^{n+1-j}} \left( \max\{1,\|u_j\|_{\rho/2}^{\rho/2}\} \right)^2 \\
	\le\; & \frac{1}{(1+\tau)^n} \|u_0\|_{m}^{m} \|u_0\|_{\infty}^{\rho-m} + \tau C(B) \rho^{(d+2)/2} U_{\rho/2}^2 \sum_{j=1}^n \frac{1}{(1+\tau)^{j}} \\
	\le\; & \frac{\|u_0\|_{m}^{m}}{(1+\tau)^n} (1+\|u_0\|_{\infty})^{\rho} + C(B) \rho^{(d+2)/2} U_{\rho/2}^2 \left[ 1 - \frac{1}{(1+\tau)^n} \right] \\
	\le\; & \max\left\{ \|u_0\|_m^{m} (1+\|u_0\|_{\infty})^{\rho} , C(B) \rho^{(d+2)/2} U_{\rho/2}^2 \right\}.
\eqntag
\label{ph45}
\end{align*}
At this point, we define $\rho_j := 2^j (1+\max\{m,d\}) > \max\{m,d\}$ for $j\ge 0$ and set $U_{j} := U_{\rho_j}$ for $j\ge 1$. 
It follows from~\eqref{ph45} with $\rho=\rho_{j+1}$ that
\begin{equation*}
	\|u_n\|_{\rho_{j+1}}^{\rho_{j+1}} \le C(B) \max\left\{ (1+\|u_0\|_\infty)^{\rho_{j+1}} , \rho_{j+1}^{(d+2)/2} U_{j}^2 \right\}, \qquad j\ge 1.
\end{equation*}
Consequently,
\begin{align*}
	U_{j+1} & \le C(B) \max\left\{ \big( 1 + \|u_0\|_\infty \big)^{\rho_{j+1}} , \rho_{j+1}^{(d+2)/2} U_{j}^2 \right\}\\
	&  \le  C(B) \rho_{j+1}^{(d+2)/2}  \max\left\{ \big( 1 + \|u_0\|_\infty \big)^{\rho_{j+1}} , U_{j}^2\right\}.
\end{align*}
We are then in a position to use \cite[Lemma~A.1]{Lau1994} (with $a=2$, $b=(d+2)/2$ and $c=0$) and conclude that
\begin{equation*}
	U_{j}^{1/\rho_{j+1}} \le C(B), \qquad j\ge 1. 
\end{equation*}
Letting $j\to\infty$ in the above inequality completes the proof.

\medskip

\noindent $\triangleright$ $d\in\{1,2\}$: Since
\begin{equation*}
	\frac{2m}{m+\rho-1} < \frac{2\rho}{m+\rho-1}<2,
\end{equation*}
we infer from H\"older's inequality that , for $n\ge 1$,
\begin{align*}
	\|u_n\|_\rho^\rho & = \left\| u_n^{(m+\rho-1)/2} \right\|_{2\rho/(m+\rho-1)}^{2\rho/(m+\rho-1)} \\
	& \le \left\| u_n^{(m+\rho-1)/2} \right\|_{2m/(m+\rho-1)}^{[2m(m-1)]/[(\rho-1)(m+\rho-1)]} \left\| u_n^{(m+\rho-1)/2} \right\|_{2}^{2(\rho-m)/(\rho-1)} \\
	& = \|u_n\|_m^{(m-1)/(\rho-1)} \left\|u_n^{(m+\rho-1)/2} \right\|_{2}^{2(\rho-m)/(\rho-1)}. 
\end{align*}
We next use the Gagliardo-Nirenberg inequality
\begin{equation*}
	\|z\|_2 \le C_{GN} \|\nabla z\|_2^{1/2} \|z\|_{2d/(d+2)}^{1/2}, \qquad z\in H^1(\mathbb{R}^d), 
\end{equation*}
and~\eqref{ph23} to obtain
\begin{align*}
	\|u_n\|_\rho^\rho & \le B^{(m-1)/(\rho-1)} C_{GN}^{2(\rho-m)/(\rho-1)} \left\| \nabla u_n^{(m+\rho-1)/2} \right\|_2^{(\rho-m)/(\rho-1)} \\
	& \hspace{2cm} \times\left[ \|u_n\|_{d(m+\rho-1)/(d+2)}^{d(m+\rho-1)/(d+2)} \right]^{[(d+2)(\rho-m)]/[2d(\rho-1)]} \\
	& \le (1+B) (1+C_{GN}^2) \left\| \nabla u_n^{(m+\rho-1)/2} \right\|_2^{(\rho-m)/(\rho-1)} \\
	& \hspace{2cm} \times\left[ \|u_n\|_{d(m+\rho-1)/(d+2)}^{d(m+\rho-1)/(d+2)} \right]^{[(d+2)(\rho-m)]/[2d(\rho-1)]}. 
\end{align*}
As $(\rho-m)/[2(\rho-1)]<1/2<1$, it follows from the above inequality and Young's inequality that
\begin{align*}
	\rho C_{1}(B) \|u_n\|_\rho^\rho & \le \frac{\rho-m}{2(\rho-1)} (m-1) M^{m-1} \left\| \nabla u_n^{(m+\rho-1)/2} \right\|_2^2 \\
	& \quad + \frac{m+\rho-2}{2(\rho-1)} \left( \frac{1}{(m-1)M^{m-1}} \right)^{(\rho-m)/(m+\rho-2)} (\rho C_2(B))^{[2(\rho-1)]/(m+\rho-2)} \\
	& \hspace{3cm} \times \left[ \|u_n\|_{d(m+\rho-1)/(d+2)}^{d(m+\rho-1)/(d+2)} \right]^{[(d+2)(\rho-m)]/[d(m+\rho-2)]}.
\end{align*}
As 
\begin{equation*}
	\frac{\rho-m}{m+\rho-2} < \frac{\rho-1}{m+\rho-2} < 1,
\end{equation*} 
we further obtain
\begin{align*}
	\rho C_{1}(B) \|u_n\|_\rho^\rho & \le \frac{(m-1) M^{m-1}}{2} \left\| \nabla u_n^{(m+\rho-1)/2} \right\|_2^2 \\
	& \quad + C_2(B) \rho^{2} \left[ \max\left\{ 1, \|u_n\|_{d(m+\rho-1)/(d+2)}^{d(m+\rho-1)/(d+2)} \right\} \right]^{(d+2)/d}.
\end{align*}
Combining the above inequality with~\eqref{ph44} leads us to 
\begin{equation}
\begin{split}
	\frac{(m-1) M^{m-1}}{2} \left\| \nabla u_n^{(m+\rho-1)/2} \right\|_2^2 & \le \frac{\|u_{n-1}\|_{\rho}^{\rho} - \|u_n\|_\rho^\rho}{\tau} - \|u_n\|_\rho^\rho \\
	& \quad + C(B) \rho^2 \left[ \max\left\{ 1, \|u_n\|_{d(m+\rho-1)/(d+2)}^{d(m+\rho-1)/(d+2)} \right\} \right]^{(d+2)/d}.
\end{split}\label{ph54}
\end{equation}
At this point, we introduce the sequence $(\rho_j)_{j\in\mathbb{N}}$ defined by
\begin{equation*}
	\rho_{j+1} = \frac{d+2}{d} \rho_j + 1 - m, \quad j\in\mathbb{N}, \qquad \rho_0 = m+d > \max\{m,d\}.
\end{equation*}
Then,
\begin{equation*}
	\rho_j = \frac{d(m-1)}{2} + \left( \frac{d+2}{d} \right)^j \left( \rho_0 - \frac{d(m-1)}{2} \right), \qquad j\in\mathbb{N},
\end{equation*}
and set
\begin{equation*}
	U_j := \sup_{n\ge 0} \max\big\{1,\|u_n\|_{\rho_j}^{\rho_j}\big\}, \qquad j\in\mathbb{N}.
\end{equation*}
With this notation, we deduce from~\eqref{ph54} with $\rho=\rho_{j+1}$, $j\in\mathbb{N}$, that
\begin{equation*}
	\|u_n\|_{\rho_{j+1}}^{\rho_{j+1}} \le \frac{1}{1+\tau} \|u_{n-1}\|_{\rho_{j+1}}^{\rho_{j+1}} + \frac{\tau}{1+\tau} C(B) \rho_{j+1}^2 U_j^{(d+2)/d}, \qquad n\ge 1.
\end{equation*}
Iterating the above formula, we find
\begin{align*}
	\|u_n\|_{\rho_{j+1}}^{\rho_{j+1}} & \le \frac{1}{(1+\tau)^n} \|u_{0}\|_{\rho_{j+1}}^{\rho_{j+1}} + \left( 1 - \frac{1}{(1+\tau)^n} \right) C(B) \rho_{j+1}^2 U_j^{(d+2)/d} \\
	& \le \frac{\|u_0\|_m^m}{(1+\tau)^n} \|u_{0}\|_{\infty}^{\rho_{j+1}-m} + \left( 1 - \frac{1}{(1+\tau)^n} \right) C(B) \rho_{j+1}^2 U_j^{(d+2)/d} \\
	& \le C(B) \max\left\{ \big(1+\|u_0\|_\infty\big)^{\rho_{j+1}} , \rho_{j+1}^2 U_j^{(d+2)/d} \right\}.
\end{align*}
Therefore, 
\begin{equation*}
	U_{j+1} \le C(B) \rho_{j+1}^2 \max\left\{ \big(1+\|u_0\|_\infty\big)^{\rho_{j+1}} , U_j^{(d+2)/d} \right\}, \qquad j\in\mathbb{N}.
\end{equation*}
We are then again in a position to apply \cite[Lemma~A.1]{Lau1994} (with $a=(d+2)/d$, $b=2$ and $c=1-m$) to conclude that
\begin{equation*}
	U_{j}^{1/\rho_{j}} \le C(B), \qquad j\in\mathbb{N}.
\end{equation*}
Letting $j\to\infty$ completes the proof, since $\rho_j\to\infty$ as $j\to\infty$.
\end{proof}

We conclude Section~\ref{sec.2} by collecting the main discrete estimates established above. As mentioned in the introduction, these discrete estimates for the minimizing movement scheme constitute the core of the paper and provide the key ingredients for the simultaneous proof of the global existence and boundedness of weak solutions.

\begin{theorem}\label{thm.est}
Let $m>1$, $M>0$ and consider initial data $(u_0,v_0,w_0)$ satisfying~\eqref{hyp_CI} and~\eqref{ph01}. There is a positive constant $C_*>0$ depending only on $d$, $m$, $M$ and $(u_0,v_0,w_0)$ such that, for any $\tau\in (0,1)$, the solution $(u_n,v_n,w_n)_{n\ge 1}$ to the time discrete scheme~\eqref{ph02} satisfies
\begin{equation*}
	\|u_n\|_m + \|u_n\|_\infty + \|w_n\|_2 + \|w_n\|_{W^{1,\infty}} + \|v_n-\Delta v_n\|_2 + \|v_n-\Delta v_n\|_\infty \le C_*, \qquad n\ge 1, 
\end{equation*}
in the following cases:
\begin{itemize}
	\item [\textbf{(m1)}] $d\ge 5$ and $m>m^* = (2d-4)/d$ or $1\le d\le 4$;
	\item [\textbf{(m2)}] $d\ge 5$, $m=m^*$ and $M\in \big(0,M^*\big)$, the threshold mass $M^*$ being defined in~\eqref{ph19ba}; 
	\item [\textbf{(m3)}] $d\ge 5$, $m\in (m_*,m^*)$ and $\|u_0\|_m^m < z_m$, where $z_m>0$ is defined by~\eqref{ph19c}. 
\end{itemize}
\end{theorem}

\begin{proof}
According to Corollary~\ref{cords04}, the property~\eqref{ph23} is satisfied with $B=A$. Theorem~\ref{thm.est} then readily follows from Corollary~\ref{cords04}, Lemma~\ref{lemds05} and Lemma~\ref{lemds06}.
\end{proof}

\section{Convergence}\label{sec.3}

Having established several uniform bounds on the sequences $\big(u_n^\tau,v_n^\tau,w_n^\tau\big)_{n\ge 1}$ in the previous section, we are in a position to prove the main result of this paper, namely the existence of a global bounded solution to~\eqref{eqn;DCI} as stated in Theorem~\ref{thm.gbe}. To this end, we shall prove the convergence of the piecewise constant interpolation $(u^\tau,v^\tau,w^\tau)$, which is defined by
\begin{equation*}
	\big(u^\tau,v^\tau,w^\tau\big)(t) := \big(u_n^\tau,v_n^\tau,w_n^\tau\big), \qquad t\in ((n-1)\tau,n\tau], \ n\ge 1, 
\end{equation*}
for $\tau\in (0,1)$. We point out here that, in view of the already obtained estimates which improve those derived in \cite[Section~5]{Mim2024b}, the convergence proof follows very closely that performed in \cite[Sections~6-7]{Mim2024b} and we just state below the needed results. We begin with the main estimates, see \cite[Proposition~5.2]{Mim2024b}, with the only difference that the $L^{p_*}\big((0,T),L^{d/(d-1)}(\mathbb{R}^d)\big)$ on $\nabla\big(u^\tau\big)^m$ therein is replaced by an estimate in $L^2\big((0,T)\times\mathbb{R}^d\big)$.

\begin{proposition}\label{prop.m1}
There exists $\kappa_0>0$ depending only on $d$, $m$, $M$, $(u_0,v_0,w_0)$ such that, for $\tau\in (0,1)$, $N\ge 1$ and $T\in [N\tau,(N+1)\tau)$,
\begin{subequations}\label{em01}
\begin{align}
	& \big\|u^\tau(t)\big\|_m^m + \big\|v^\tau(t)\big\|_{W^{2,2}}^2 + \big\|w^\tau(t)\big\|_{W^{1,2}}^2 \le \kappa_0, \label{em01a}\\
	& \int_{\mathbb{R}^d} |x|^2 u^\tau(t,x)\ dx \le \kappa_0 (1+T), \qquad t\in [0,T], \label{em01b}\\
	& \int_0^T \left( \big\|u^\tau(t)\big\|_2^2 + \big\|\nabla \big(u^\tau\big)^m(t)\big\|_2^2 + \big\|v^\tau(t)\big\|_{W^{3,2}}^2 + \big\|w^\tau(t)\big\|_{W^{2,2}}^2\right)\ dt \le \kappa_0 (1+T), \label{em01c}\\
	& \sum_{n=1}^N \frac{\mathcal{W}_2^2\big(u_n^\tau,u_{n-1}^\tau\big)}{\tau} + \sum_{n=1}^N \frac{\big\|v_n^\tau-v_{n-1}^\tau\big\|_{W^{1,2}}^2}{\tau} + \sum_{n=1}^N \frac{\big\|w_n^\tau-w_{n-1}^\tau\big\|_{2}^2}{\tau} \le \kappa_0. \label{em01d}
\end{align}	
\end{subequations}
\end{proposition}

\begin{proof}
Let $\tau\in (0,1)$. By Theorem~\ref{thm.est}, there is $C_*>0$ depending only on $d$, $m$, $M$ and $(u_0,v_0,w_0)$ such that
\begin{equation}
	\big\|u_n^\tau\big\|_m^m + \big\|v_n^\tau\big\|_{W^{2,2}}^2 + \big\|w_n^\tau\big\|_2^2 + \big\|u_n^\tau\big\|_2^2 \le C_*, \qquad n\ge 1. \label{em02}
\end{equation}
Combining~\eqref{em02} and Lemma~\ref{lemA2} (with $z_n=w_n$ and $f_n=u_{n-1}$) gives
\begin{equation}
	\big\|w_n^\tau\big\|_{W^{1,2}}^2 \le \kappa_0, \qquad n\ge 1, \label{em03}
\end{equation}
and we have proved~\eqref{em01a}. The proofs of~\eqref{em01b} and~\eqref{em01d} are the same as that in \cite[Proposition~5.2]{Mim2024b}, to which we refer. 

Next, since $\rho=m+1>m$ and the proof of \textbf{Step~1} in Lemma~\ref{lemds04} does not require the additional constraint $\rho>(d-2)/2$, it follows from~\eqref{ph29} (with $\rho=m+1$) that
\begin{equation*}
	\frac{(m+1) M^{m-1}}{m} \big\|\nabla \big(u_n^\tau\big)^{m}\big\|_2^2 \le \frac{\|u_{n-1}\|_{m+1}^{m+1} - \|u_n\|_{m+1}^{m+1}}{m \tau} + M\int_{\mathbb{R}^d} \big(u_n^\tau\big)^{m+1} \big(v_n^\tau - \Delta v_n^\tau\big)\ dx. 
\end{equation*}
We next use H\"older's inequality, Corollary~\ref{cords04}, Lemma~\ref{lemds05} and Lemma~\ref{lemds06} to obtain
\begin{align*}
	\frac{(m+1) M^{m-1}}{m} \big\|\nabla \big(u_n^\tau\big)^{m}\big\|_2^2 & \le \frac{\|u_{n-1}\|_{m+1}^{m+1} - \|u_n\|_{m+1}^{m+1}}{m \tau} + M\int_{\mathbb{R}^d} \big(u_n^\tau\big)^{m+1} \big(v_n^\tau - \Delta v_n^\tau\big)\ dx \\
	& \le \frac{\|u_{n-1}\|_{m+1}^{m+1} - \|u_n\|_{m+1}^{m+1}}{m \tau} + M \big\|u_n^\tau\big\|_\infty \big\|u_n^\tau\big\|_m^m \big\|v_n^\tau - \Delta v_n^\tau\big\|_\infty  \\
	& \le \frac{\|u_{n-1}\|_{m+1}^{m+1} - \|u_n\|_{m+1}^{m+1}}{m \tau} + \kappa_0.
\end{align*}
Summing up the above inequality for $n\in\{1,\ldots,N\}$, $N\ge 1$, leads us to
\begin{equation}
	\frac{(m+1) M^{m-1}}{m} \tau \sum_{n=1}^N \big\|\nabla \big(u_n^\tau\big)^{m}\big\|_2^2 \le \frac{\|u_0\|_{m+1}^{m+1}}{m} + N \tau \kappa_0 \le \kappa_0 (1+N\tau). \label{em04}
\end{equation}
Finally, we infer from~\eqref{em02}, \eqref{em03} and Lemma~\ref{lemA2} (first, with $z_n=\nabla v_n^\tau$ and $f_n=\nabla w_n^\tau$ and then with $z_n=w_n^\tau$ and $f_n=u_{n-1}^\tau$) that
\begin{equation*}
	\tau \sum_{n=1}^N \big\| \big( \partial_i\partial_j v_n^\tau \big)\big\|_{W^{1,2}}^2 + \tau \sum_{n=1}^N \big\| \nabla w_n^\tau\big\|_{W^{1,2}}^2 \le \kappa_0 (1+N\tau).
\end{equation*}
The bound~\eqref{em01c} is then a direct consequence of~\eqref{em02}, \eqref{em04} and the above estimate.
\end{proof}

The analysis performed in Section~\ref{sec.2} actually goes beyond the outcome of Proposition~\ref{prop.m1} and we report the following additional estimates, which are established in Lemma~\ref{lemds05} and Lemma~\ref{lemds06} and summarized in Theorem~\ref{thm.est}.

\begin{proposition}\label{prop.eh1}
There is $\kappa_\infty>0$ depending on $d$, $m$, $M$ and $(u_0,v_0,w_0)$ such that
\begin{equation}
	\big\|u^\tau(t)\big\|_\infty + \big\|v^\tau(t)\big\|_{W^{2,\infty}} + \big\|w^\tau(t)\big\|_{W^{1,\infty}} \le \kappa_\infty. \label{ehl01}
\end{equation}
\end{proposition}

We then argue as in \cite[Section~6]{Mim2024b} to deduce from Proposition~\ref{prop.m1} and a refined version of the Ascoli-Arzel\`a theorem \cite[Theorem~3.3.1]{AGS2005} the following convergence result, see also the proof of \cite[Eq.~(30)]{JKO1998}.

\begin{proposition}\label{prop.m2}
There is a sequence $(\tau_i)_{i\ge 1}$ in $(0,1)$ and functions
\begin{equation*}
	u\in C\big([0,\infty),\mathcal{P}_{2}(\mathbb{R}^d)\big) , \quad v\in C\big([0,\infty),W^{1,2}(\mathbb{R}^d)\big), \quad w\in C\big([0,\infty),L^{2}(\mathbb{R}^d)\big)
\end{equation*}
such that, for any $t>0$, 
\begin{equation*}
	\lim_{i\to\infty} \tau_i = \lim_{i\to\infty} \left[  \mathcal{W}_2^2\big(u^{\tau_i}(t),u(t)\big) + \big\| v^{\tau_i}(t)-v(t)\big\|_{W^{1,2}} + \big\| w^{\tau_i}(t)-w(t)\big\|_{2} \right] = 0 
\end{equation*}
and
\begin{equation}
\begin{split}
	u^{\tau_i}(t) & \rightharpoonup u(t) \;\;\text{ in }\;\; L^1(\mathbb{R}^d)\cap L^m(\mathbb{R}^d), \\
	v^{\tau_i}(t) & \rightharpoonup v(t) \;\;\text{ in }\;\; W^{2,2}(\mathbb{R}^d), \\
	w^{\tau_i}(t) & \rightharpoonup w(t) \;\;\text{ in }\;\; W^{1,2}(\mathbb{R}^d).
\end{split}\label{cv02}
\end{equation}
Moreover, 
\begin{equation*}
	\lim_{i\to\infty} u^{\tau_i}(t,x) = u(t,x) \;\text{ for a.e. }\; (t,x)\in (0,\infty)\times\mathbb{R}^d.
\end{equation*}
\end{proposition}

Since both $u^{\tau_i}(t)$ and $u(t)$ are probability measures for all $t\ge 0$ and $i\ge 1$, a straightforward consequence of~\eqref{ehl01} and~\eqref{cv02} is the following.

\begin{corollary}\label{corReg}
	For any $\rho\in [1,\infty)$, $u\in L^\infty((0,\infty),L^\rho(\mathbb{R}^d))$ and $u\in L^\infty((0,\infty)\times \mathbb{R}^d)$.
\end{corollary}

With the above stated results at hand, we are ready to complete the proof of Theorem~\ref{thm.gbe}. For that purpose, we proceed as in \cite[Section~7]{Mim2024b}, to which we refer. 

\section*{Acknowledgments}\label{sec.Acknowledgments}
The work of the first author is partially supported by JSPS KAKENHI Grant Number 25K17274 (Early-Career Scientists), 25KJ0279 (Research Fellowship for Young Scientists), and the French Government Scholarship Program, Bourses France Excellence Japon 2025. The first author thanks Laboratoire de Math\'ematiques LAMA in Chambéry and Institut de Math\'ematiques de Toulouse, where part of this work was done, for their hospitality.

\appendix
\section{Estimates for solutions to time discrete heat equations}\label{sec.A}

In this appendix, we collect several auxiliary estimates for solutions to time discrete linear heat equations that are repeatedly used throughout Section~\ref{sec.2}. We first establish discrete $L^{p}$-maximal regularity estimates with nonzero initial data. We then derive elementary $L^2$-energy estimates and discrete counterparts of the classical smoothing properties of the heat equation. All constants appearing below are independent of the time step $\tau\in (0,1)$.

\begin{lemma}\label{lemds02a}
	Consider $(q,s_0)\in (1,\infty)^2$ and two sequences $(f_n)_{n\ge 1}$ and $(z_n)_{n\in\mathbb{N}}$ of functions belonging to $L^{s_0}(\mathbb{R}^d)$ and $W^{2,s_0}(\mathbb{R}^d)$, respectively, which satisfy
	\begin{equation}
		\frac{z_n-z_{n-1}}{\tau} - \Delta z_n + z_n = f_n \;\;\text{ in }\;\;\mathbb{R}^d, \label{zLs00}
	\end{equation}
	for $n\ge 1$. Then, for all $N\ge 1$, 
	\begin{equation}
		\left( \sum_{n=1}^N \|z_n\|_{s_0}^q \right)^{1/q} \le \left( \sum_{n=1}^N \|f_n\|_{s_0}^q \right)^{1/q} + \tau^{-1/q} \|z_0\|_{s_0}, \label{zLs01} 
	\end{equation}
	and there is a positive constant $C_0(q,s_0)>0$ depending only on $d$, $q$ and $s_0$ such that 
	\begin{equation}
		\begin{split}
			\frac{1}{\tau} \left( \sum_{n=1}^N \|z_n - z_{n-1}\|_{s_0}^q \right)^{1/q} & +  \left( \sum_{n=1}^N \|z_n - \Delta z_n\|_{s_0}^q \right)^{1/q} \\
			& \le C_0(q,s_0) \left[\left( \sum_{n=1}^N \|f_n\|_{s_0}^q \right)^{1/q} + \tau^{-1/q} \|z_0-\Delta z_0\|_{s_0} \right].
		\end{split}  \label{zLs02} 
	\end{equation}	
	Moreover,
	\begin{equation}
		\|z_n\|_{s_0} \le \max\left\{ \|z_0\|_{s_0} , \max_{1\le j\le n} \|f_j\|_{s_0} \right\}, \qquad n\ge 1. \label{zLs03} 
	\end{equation}
\end{lemma}

\begin{proof}
	Let $n\ge 1$. We infer from~\eqref{zLs00} and H\"older's inequality that
	\begin{align*}
		(1+\tau) \|z_n\|_{s_0}^{s_0} & \le (1+\tau) \|z_n\|_{s_0}^{s_0} + \tau (s_0-1) \int_{\mathbb{R}^d} |z_n|^{s_0-2} |\nabla z_n|^2 \, \mathrm{d}x \\
		& = \int_{\mathbb{R}^d} |z_n|^{s_0-2} z_n \left[ (1+\tau) z_n - \Delta z_n \right]\, \mathrm{d}x \\
		& = \int_{\mathbb{R}^d} |z_n|^{s_0-2} z_n \left( z_{n-1} + \tau f_n \right)\, \mathrm{d}x \\
		& \le \|z_n\|_{s_0}^{s_0-1} \left( \|z_{n-1}\|_{s_0} + \tau \|f_n\|_{s_0} \right), 
	\end{align*}
	whence
	\begin{equation}
		\|z_n\|_{s_0} \le \frac{1}{1+\tau} \|z_{n-1}\|_{s_0} + \frac{\tau}{1+\tau} \|f_n\|_{s_0}, \qquad n\ge 1. \label{zLs04}
	\end{equation}
	Iterating~\eqref{zLs04}, we obtain
	\begin{align*}
		\|z_n\|_{s_0} & \le \frac{1}{(1+\tau)^n} \|z_0\|_{s_0} + \frac{\tau}{1+\tau} \sum_{j=1}^n \frac{\|f_j\|_{s_0}}{(1+\tau)^{n-j}} \\
		& \le \frac{1}{(1+\tau)^n} \|z_0\|_{s_0} + \frac{\tau}{(1+\tau)^{n+1}} \left( \sum_{j=1}^n (1+\tau)^j \right) \max_{1\le j\le n} \|f_j\|_{s_0} \\
		& = \frac{1}{(1+\tau)^n} \|z_0\|_{s_0} + \left( 1 - \frac{1}{(1+\tau)^n} \right) \max_{1\le j\le n} \|f_j\|_{s_0}, 
	\end{align*}
	from which~\eqref{zLs03} follows.
	
	We next split $z_n$ into the uniquely defined functions $\widehat{z}_n$ and $\widetilde{z}_n$ satisfying $z_n=\widehat{z}_n+\widetilde{z}_n$ and
	\begin{equation}
		\begin{split}
			\left\{
			\begin{aligned}
				&\frac{\widehat{z}_n-\widehat{z}_{n-1}}{\tau}-\Delta \widehat{z}_n+\widehat{z}_n=f_{n} \;\;\text{ in }\;\;\mathbb{R}^d,
				\\
				&\widehat{z}_0=0 \;\;\text{ in }\;\;\mathbb{R}^d,
			\end{aligned}
			\right.
			\quad
			\left\{
			\begin{aligned}
				&\frac{\widetilde{z}_n-\widetilde{z}_{n-1}}{\tau}-\Delta \widetilde{z}_n+\widetilde{z}_n=0 \;\;\text{ in }\;\;\mathbb{R}^d,
				\\
				&\widetilde{z}_0=z_0 \;\;\text{ in }\;\;\mathbb{R}^d.
			\end{aligned}
			\right.
		\end{split} \label{zLs05}
	\end{equation}
	On the one hand, according to~\eqref{zLs05}, $(\widehat{z}_n)_{n\ge 1}$ solves the same system as $(z_n)_{n\ge 1}$ but with initial condition $\widehat{z}_0=0$, so that it satisfies~\eqref{zLs04}; that is,
	\begin{equation*}
		\|\widehat{z}_n\|_{s_0} \le \frac{1}{1+\tau} \|\widehat{z}_{n-1}\|_{s_0} + \frac{\tau}{1+\tau} \|f_n\|_{s_0}, \qquad n\ge 1. 
	\end{equation*}
	Therefore,
	\begin{align*}
		\left( \sum_{n=1}^N \big\|\widehat{z}_n\big\|_{s_0}^q \right)^{1/q} & \le \frac{1}{1+\tau} \left( \sum_{n=1}^N \big\|\widehat{z}_{n-1}\big\|_{s_0}^q  \right)^{1/q} + \frac{\tau}{1+\tau} \left( \sum_{n=1}^N \|f_n\|_{s_0}^q  \right)^{1/q} \\
		& = \frac{1}{1+\tau} \left( \sum_{n=1}^{N-1} \big\|\widehat{z}_{n}\big\|_{s_0}^q  \right)^{1/q} + \frac{\tau}{1+\tau} \left( \sum_{n=1}^N \|f_n\|_{s_0}^q  \right)^{1/q} ,
	\end{align*}
	due to $\widehat{z}_0=0$, from which we deduce that
	\begin{equation}
		\left( \sum_{n=1}^N \big\|\widehat{z}_n\big\|_{s_0}^q \right)^{1/q}  \le \left( \sum_{n=1}^N \|f_n\|_{s_0}^q  \right)^{1/q} . \label{zLs06}
	\end{equation}
	Furthermore, it follows from the scheme~\eqref{zLs05} for $\widehat{z}_n$, the maximal regularity of $I-\Delta$ in $L^{s_0}(\mathbb{R}^d)$ and \cite[Theorem~3.1]{KLL2016}, see also \cite[Remark~5.2]{APW2002}, that there is $C_{MR}(q,s_0)>0$ depending only on $d$, $s_0$ and $q$ such that, for all $N\ge 1$,
	\begin{equation}
		\begin{split}
			\frac{1}{\tau} \left( \sum_{n=1}^N \big\| \widehat{z}_n - \widehat{z}_{n-1}\big\|_{s_0}^q \right)^{1/q} & + \left( \sum_{n=1}^N \left\| \widehat{z}_n - \Delta \widehat{z}_n \right\|_{s_0}^q \right)^{1/q}  \le C_{MR}(q,s_0) \left( \sum_{n=1}^N \big\|f_{n} \big\|_{s_0}^q \right)^{1/q}.
		\end{split}\label{zLs07}
	\end{equation}
	On the other hand, we claim that
	\begin{equation}
		\left\|\widetilde{z}_n\right\|_{s_0} \le \frac{\|z_0\|_{s_0}}{(1+\tau)^n}, \qquad n\in\mathbb{N}. \label{zLs08}
	\end{equation}
	Indeed, for $n\ge 1$, the scheme~\eqref{zLs05} for $\widetilde{z}_n$ reads
	\begin{equation*}
		(1+\tau) \widetilde{z}_n - \tau\Delta\widetilde{z}_n = \widetilde{z}_{n-1} \;\;\text{ in }\;\;\mathbb{R}^d.
	\end{equation*}
	Multiplying both sides of the above identity by $|\widetilde{z}_n|^{s_0-2} \widetilde{z}_n$ and integrating over $\mathbb{R}^d$ give
	\begin{equation*}
		(1+\tau) \big\|\widetilde{z}_n\big\|_{s_0}^{s_0} - \tau \int_{\mathbb{R}^d} |\widetilde{z}_n|^{s_0-2} \widetilde{z}_n \Delta\widetilde{z}_n \, dx = \int_{\mathbb{R}^d} |\widetilde{z}_n|^{s_0-2} \widetilde{z}_n \widetilde{z}_{n-1} \, dx.
	\end{equation*}
	By integration by parts and H\"older's inequality, we further obtain
	\begin{equation*}
		(1+\tau) \big\|\widetilde{z}_n\big\|_{s_0}^{s_0} \le (1+\tau) \big\|\widetilde{z}_n\big\|_{s_0}^{s_0} + \tau (s_0-1) \int_{\mathbb{R}^d} |\widetilde{z}_n|^{s_0-2} |\nabla\widetilde{z}_n|^2 \, dx \le \big\|
		\widetilde{z}_n\big\|_{s_0}^{s_0-1} \big\|\widetilde{z}_ {n-1}\big\|_{s_0},
	\end{equation*}
	from which we deduce that
	\begin{equation*}
		(1+\tau) \big\|\widetilde{z}_n\big\|_{s_0} \le \big\|\widetilde{z}_{n-1}\big\|_{s_0}, \qquad n\ge 1.
	\end{equation*}
	Iterating the above estimate gives the claim~\eqref{zLs08} as $\widetilde{z}_0=z_0$. Now, for $N\ge 1$, we compute the sum
	\begin{align*}
		\left( \sum_{n=1}^N \frac{1}{(1+\tau)^{nq}} \right)^{1/q} & = \left( \frac{(1+\tau)^{-q} \bigl(1-(1+\tau)^{-Nq}\bigr)}{1-(1+\tau)^{-q}}\right)^{1/q} \\
		& = \left( \frac{1 - (1+\tau)^{-Nq}}{(1+\tau)^{q}-1} \right)^{1/q}
	\end{align*}
	and obtain, using $1-(1+\tau)^{-Nq}\le1$ and  $(1+\tau)^{q}-1\ge q\tau$, 
	\begin{equation*}
		\left( \sum_{n=1}^N \frac{1}{(1+\tau)^{nq}} \right)^{1/q} \le (q\tau)^{-1/q} \le \tau^{-1/q}.
	\end{equation*}
	Together with~\eqref{zLs08}, the above inequality gives
	\begin{equation}
		\left( \sum_{n=1}^N \big\|\widetilde{z}_n\big\|_{s_0}^q \right)^{1/q} \le \tau^{-1/q} \|z_0\|_{s_0}. \label{zLs09}
	\end{equation}
	Observing that $\widetilde{z}_n - \Delta\widetilde{z}_n$ satisfies the same equation as $\widetilde{z}_n$ for $n\ge 1$ but with initial condition $z_0-\Delta z_0\in L^{s_0}(\mathbb{R}^d)$, we readily conclude from~\eqref{zLs09} that
	\begin{equation}
		\left( \sum_{n=1}^N \big\|\widetilde{z}_n - \Delta\widetilde{z}_n \big\|_{s_0}^q \right)^{1/q} \le \tau^{-1/q} \|z_0 - \Delta z_0\|_{s_0}, \qquad N\ge 1. \label{zLs10}
	\end{equation}
	
	We now combine~\eqref{zLs06} and~\eqref{zLs09} to obtain
	\begin{align*}
		\left( \sum_{n=1}^N \|z_n\|_{s_0}^q \right)^{1/q} & \le \left( \sum_{n=1}^N \big\|\widehat{z}_n\big\|_{s_0}^q \right)^{1/q}  + \left( \sum_{n=1}^N \big\|\widetilde{z}_n\big\|_{s_0}^q \right)^{1/q} \\
		& \le \left( \sum_{n=1}^N \|f_n\|_{s_0}^q \right)^{1/q} + \tau^{-1/q} \|z_0\|_{s_0},
	\end{align*}
	thus establishing~\eqref{zLs01}. Similarly, we infer from~\eqref{zLs05}, \eqref{zLs07} and~\eqref{zLs10} that
	\begin{align*}
		& \frac{1}{\tau} \left( \sum_{n=1}^N \|z_n - z_{n-1}\|_{s_0}^q \right)^{1/q} + \left( \sum_{n=1}^N \|z_n - \Delta z_n\|_{s_0}^q \right)^{1/q} \\
		& \quad \le \frac{1}{\tau} \left( \sum_{n=1}^N \big\|\widehat{z}_n - \widehat{z}_{n-1}\|_{s_0}^q \right)^{1/q} + \left( \sum_{n=1}^N \big\|\widehat{z}_n - \Delta\widehat{z}_n \big\|_{s_0}^q \right)^{1/q} \\
		&\quad\quad + \frac{1}{\tau} \left( \sum_{n=1}^N \big\|\widetilde{z}_n - \widetilde{z}_{n-1}\|_{s_0}^q \right)^{1/q} + \left( \sum_{n=1}^N \big\|\widetilde{z}_n - \Delta\widetilde{z}_n \big\|_{s_0}^q \right)^{1/q} \\
		& \quad\le C_{MR}(q,s_0)  \left( \sum_{n=1}^N \|f_n\|_{s_0}^q \right)^{1/q} + 2 \left( \sum_{n=1}^N \big\|\widetilde{z}_n - \Delta\widetilde{z}_n \big\|_{s_0}^q \right)^{1/q} \\
		& \quad\le C_{MR}(q,s_0)  \left( \sum_{n=1}^N \|f_n\|_{s_0}^q \right)^{1/q} + 2 \tau^{-1/q} \|z_0-\Delta z_0\|_{s_0},
	\end{align*}
	which proves~\eqref{zLs02} with $C_0(q,s_0) := \max\big\{ C_{MR}(q,s_0) , 2 \big\}$ and completes the proof of Lemma~\ref{lemds02a}.
\end{proof}

We next establish a basic $L^2$-energy estimate for solutions to the time discrete heat equation.

\begin{lemma}\label{lemA2}
Consider two sequences $(f_n)_{n\ge 1}$ and $(z_n)_{n\in\mathbb{N}}$ of functions belonging to $L^{2}(\mathbb{R}^d)$ and $W^{2,2}(\mathbb{R}^d)$, respectively, which satisfy~\eqref{zLs00} for $n\ge 1$ and
\begin{equation}
	F_2 := \sup_{n\ge 1} \|f_n\|_2 < \infty. \label{zLs11}
\end{equation} 
Then, 
\begin{equation}
\begin{split}
	\|z_n\|_{W^{1,2}} & \le \max\left\{ \|z_0\|_{W^{1,2}} , F_2 \right\}, \qquad n\in\mathbb{N}, \\
	\tau \sum_{j=1}^n \left( \|\Delta z_j\|_2^2 + \|\nabla z_j\|_{2}^2 \right)  & \le \|\nabla z_{0}\|_2^2 + n\tau F_2^2, \qquad n\ge 1.
\end{split}\label{zLs12} 
\end{equation}
\end{lemma}

\begin{proof}
It first follows from~\eqref{zLs03} (with $s_0=2$) that
\begin{equation}
	\|z_n\|_2 \le \max\left\{ \|z_0\|_2 , F_2 \right\}, \qquad n\ge 1. \label{zLs13}
\end{equation}
Next, for $n\ge 1$, we multiply~\eqref{zLs00} by $-\Delta z_n$ and integrate over $\mathbb{R}^d$ to obtain
\begin{equation*}
	\int_{\mathbb{R}^d} \nabla(z_n - z_{n-1})\cdot \nabla z_n\ dx + \tau \|\Delta z_n\|_2^2 + \tau \|\nabla z_n\|_2^2 = - \tau \int_{\mathbb{R}^d} f_n \Delta z_n\ dx,
\end{equation*}
from which we deduce, thanks to the Cauchy-Schwarz inequality,
\begin{equation*}
	\frac{\|\nabla z_n\|_2^2 - \|\nabla z_{n-1}\|_2^2}{2} + \tau \|\Delta z_n\|_2^2 + \tau \|\nabla z_n\|_2^2 \le \frac{\tau}{2} \left( \|\Delta z_n\|_2^2 + ||f_n\|_2^2 \right).   
\end{equation*}
We then infer from~\eqref{zLs11} and the above inequality that
\begin{equation}
	(1+\tau) \|\nabla z_n\|_2^2 + \tau \|\Delta z_n\|_2^2 + \tau \|\nabla z_n\|_2^2 \le \|\nabla z_{n-1}\|_2^2 + \tau F_2^2, \qquad n\ge 1. \label{zLs14}
\end{equation}
A first consequence of~\eqref{zLs14} is that $(\nabla z_n)_{n\ge 1}$ satisfies
\begin{equation*}
	\|\nabla z_n\|_2^2 \le \frac{1}{1+\tau} \|\nabla z_{n-1}\|_2^2 + \frac{\tau}{1+\tau} F_2^2, \qquad n\ge 1,
\end{equation*}
from which we deduce
\begin{equation}
	\|\nabla z_n\|_2^2 \le \frac{1}{(1+\tau)^n} \|\nabla z_0\|_2^2 + \left( 1 - \frac{1}{(1+\tau)^n} \right) F_2^2 \le \max\left\{ \|\nabla z_0\|_2 , F_2^2 \right\},  \qquad n\ge 1. \label{zLs15}
\end{equation}
Gathering~\eqref{zLs13} and~\eqref{zLs15} provides the first estimate in~\eqref{zLs12}. We next sum up~\eqref{zLs14} with respect to $n\in\{1,\dots,N\}$, $N\ge 1$, and thereby find
\begin{equation*}
	(1+\tau) \sum_{n=1}^N \|\nabla z_n\|_2^2 + \tau \sum_{n=1}^N \left( \|\Delta z_n\|_2^2 + \|\nabla z_n\|_{2}^2 \right) \le \sum_{n=1}^N \|\nabla z_{n-1}\|_2^2 + N\tau F_2^2,
\end{equation*}
whence
\begin{equation*}
	\tau \sum_{n=1}^N \left( \|\Delta z_n\|_2^2 + \|\nabla z_n\|_{2}^2 \right)  \le \|\nabla z_{0}\|_2^2 + N\tau F_2^2,
\end{equation*}
which completes the proof.
\end{proof}

We finally provide smoothing properties at the discrete level, which are used in the proof of Lemma~\ref{lemds05}, to improve the regularity on $(w_n)_{n\ge1}$ and $(v_n)_{n\ge1}$.

\begin{lemma}\label{lemA3}
Consider $p\in (d,\infty)$ and two sequences $(f_n)_{n\ge 1}$ and $(z_n)_{n\in\mathbb{N}}$ belonging to $L^p(\mathbb{R}^d)$ and $W^{2,p}(\mathbb{R}^d)$, respectively, which satisfy~\eqref{zLs00} for $n\ge 1$. If 
\begin{equation}
	F_p := \sup_{n\ge 1} \|f_n\|_p < \infty, \label{zLs16a}
\end{equation}
then there is a positive constant $C_1(p)>0$ depending only on $d$ and $p$ such that
\begin{equation}
	\|z_n\|_{W^{1,\infty}} \le C_1(p) \left( \|z_0\|_{W^{1,\infty}} + F_p \right), \qquad n\in\mathbb{N}. \label{zLs17a}
\end{equation}
Assume next that $(f_n)_{n\ge 1}$ is a sequence in $W^{1,\infty}(\mathbb{R}^d)$ such that
\begin{equation}
	F_{1,\infty} := \sup_{n\ge 1} \|f_n\|_{W^{1,\infty}} < \infty. \label{zLs16b}
\end{equation}
Then there is a positive constant $C_2>0$ such that
\begin{equation}
	\|z_n-\Delta z_n\|_\infty \le C_2 \left( \|z_0-\Delta z_0\|_\infty + F_{1,\infty} \right), \qquad n\ge 1. \label{zLs17b}
\end{equation}
\end{lemma}

\begin{proof}
Since $A=\mathrm{id} - \Delta$ is the infinitesimal generator of a semigroup of contractions in $L^\rho(\mathbb{R}^d)$ for all $\rho\in [1,\infty)$, the resolvent $R(\lambda,A) := (\lambda + A)^{-1}$ can be expressed in terms of the corresponding semigroup $\{e^{tA}\}_{t\ge 0}$ according to the following formula
\begin{equation*}
	R(\lambda,A) = \int_0^\infty e^{-\lambda t} e^{tA}\ dt, \qquad \lambda>0, 
\end{equation*}
see \cite[Theorem~II.1.10]{EnNa2000} for instance, and a similar formula is available for its powers \cite[Corollary~II.1.11]{EnNa2000}
\begin{equation}
	R^n(\lambda,A) = \int_0^\infty \frac{t^{n-1}}{\Gamma(n)} e^{-\lambda t} e^{tA}\ dt, \qquad \lambda>0, \ n\ge 1. \label{zLs18}
\end{equation} 
With this notation, an equivalent formulation of~\eqref{zLs00} reads
\begin{equation*}
	z_n = \frac{1}{\tau} R\left( \frac{1}{\tau},A\right) z_{n-1} + R\left( \frac{1}{\tau},A\right) f_{n}, \qquad n\ge 1, 
\end{equation*}
from which we deduce by induction that
\begin{equation}
	z_n = \frac{1}{\tau^n} R^n\left( \frac{1}{\tau},A\right) z_{0} + \tau \sum_{k=1}^n \frac{1}{\tau^{n+1-k}} R^{n+1-k}\left( \frac{1}{\tau},A\right) f_{k}, \qquad n\ge 1. \label{zLs19}
\end{equation}
At this point, we recall that the semigroup $\{e^{tA}\}_{t\ge 0}$ is explicit in that case and is given by
\begin{equation*}
	e^{tA} z = e^{-t} G_t* z, \quad G_t(x) := (4\pi t)^{-d/2} e^{-|x|^2/(4t)}, \qquad (t,x)\in (0,\infty)\times\mathbb{R}^d.
\end{equation*}
Combining~\eqref{zLs18} and~\eqref{zLs19}, we obtain the representation formula
\begin{align*}
	\frac{1}{\tau^n} R^n\left( \frac{1}{\tau},A\right) z =\; & \frac{1}{\tau^n} \int_0^\infty  \frac{t^{n-1}}{\Gamma(n)} e^{-(1+\tau)t/\tau} G_t*z\ dt = \int_0^\infty  \frac{t^{n-1}}{\Gamma(n)} e^{-(1+\tau)t} G_{\tau t}*z\ dt \\
	=\; & \mathcal{B}_{\tau}^n * z, \qquad n\ge 1,
\end{align*}
with
\begin{equation}
	\mathcal{B}_{\tau}^\alpha(x) := \int_0^\infty \frac{t^{\alpha-1}}{\Gamma(\alpha)} e^{-(1+\tau)t} G_{\tau t}(x)\ dt, \qquad x\in\mathbb{R}^d, \ \alpha>1. \label{zLs22}
\end{equation}
Plugging~\eqref{zLs22} in~\eqref{zLs19} leads us to 
\begin{equation}
	z_n = \mathcal{B}_\tau^n * z_0 + \tau \sum_{k=1}^n \mathcal{B}_\tau^{n+1-k} * f_k = \mathcal{B}_\tau^n * z_0 + \tau \sum_{k=1}^n \mathcal{B}_\tau^k * f_{n+1-k}, \qquad n\ge 1. \label{zLs23}
\end{equation}
We next recall that the self-similarity of the heat kernel ensures that, for $\rho\in [1,\infty]$ and $\alpha>1$,
\begin{subequations}\label{zLs24}
\begin{align}
	\|\mathcal{B}_\tau^\alpha\|_\rho & \le\; \frac{\Gamma\left( \alpha - \frac{d(\rho-1)}{2\rho} \right)}{\Gamma(\alpha)} \frac{\|G_1\|_\rho}{(1+\tau)^{\alpha}} \left( \frac{1+\tau}{\tau} \right)^{d(\rho-1)/(2\rho)}, \qquad \alpha > \frac{d}{2}\left( 1 - \frac{1}{\rho}\right), \label{zLs24a}\\
	\|\nabla\mathcal{B}_\tau^\alpha\|_\rho & \le\; \frac{\Gamma\left( \alpha - \frac{1}{2} - \frac{d(\rho-1)}{2\rho} \right)}{\Gamma(\alpha)} \frac{\|\nabla G_1\|_\rho}{(1+\tau)^{\alpha}} \left( \frac{1+\tau}{\tau} \right)^{[\rho+d(\rho-1)]/(2\rho)}, \qquad \alpha > \frac{1}{2} + \frac{d}{2}\left( 1 - \frac{1}{\rho}\right). \label{zLs24b}
\end{align}
\end{subequations}
Pick now $p\in (d,\infty)$ and denote its conjugate by $p'=p/(p-1)$. Then, for $n\ge 1$ and $1\le k \le n$, 
\begin{equation}
	k - \frac{1}{2} - \frac{d}{2}\left( 1 - \frac{1}{p'}\right) \ge 1 - \frac{1}{2} - \frac{d}{2p} = \frac{p-d}{2p}>0, \label{zLs25}
\end{equation}
and we infer from~\eqref{zLs16a}, \eqref{zLs23}, \eqref{zLs24} and Young's inequality (for convolutions) that
\begin{align*}
	\|z_n\|_\infty \le\; & \|\mathcal{B}_\tau^n\|_1 \|z_0\|_\infty + \tau \sum_{k=1}^{n} \|\mathcal{B}_\tau^{k}\|_{p'} \| f_{n+1-k}\|_p \\
	\le\; & \frac{\|G_1\|_1}{(1+\tau)^n} \|z_0\|_\infty + \|G_1\|_{p'} F_p \left( \frac{\tau}{1+\tau} \right)^{(2p-d)/(2p)} \sum_{k=1}^n \frac{1}{(1+\tau)^{k-1}} \frac{\Gamma\left(k - \frac{d}{2p}\right)}{\Gamma(k)}.
	\eqntag
	\label{zLs26}
\end{align*}
To estimate the sum on the right-hand side of~\eqref{zLs26}, we use Gautschi's inequality
\begin{equation*}
	\frac{\Gamma(y+s)}{\Gamma(y+1)} \le y^{s-1}, \qquad (s,y)\in (0,1)\times (0,\infty),
\end{equation*}
to obtain
\begin{align*}
	\sum_{k=1}^n \frac{1}{(1+\tau)^{k-1}} \frac{\Gamma\left(k - \frac{d}{2p}\right)}{\Gamma(k)} \le\; & \Gamma\left( \frac{2p-d}{2p} \right) + \sum_{k=2}^{n} \frac{1}{(1+\tau)^{k-1}} \frac{\Gamma\left(k-1 + 1 - \frac{d}{2p}\right)}{\Gamma(k-1+1)} \\
	\le\; & \Gamma\left( \frac{2p-d}{2p} \right) + \sum_{k=2}^{n} \frac{1}{(1+\tau)^{k-1}}\frac{1}{(k-1)^{d/(2p)}} \\
	=\,& \Gamma\left( \frac{2p-d}{2p} \right) + \sum_{k=1}^{n-1} \frac{1}{(1+\tau)^{k}}\frac{1}{k^{d/(2p)}} \\
	\le\; & \Gamma\left( \frac{2p-d}{2p} \right) + \int_0^{\infty}  \frac{\sigma^{-d/(2p)}}{(1+\tau)^{\sigma}}\ d\sigma \\
	=\; & \Gamma\left( \frac{2p-d}{2p} \right) \left( 1 + [\ln{(1+\tau)}]^{(d-2p)/(2p)} \right).
\end{align*}
Recalling that $\ln{(1+\tau)}\ge\tau/2$ as $\tau\in (0,1)$, we further obtain
\begin{align*}
	\sum_{k=1}^n \frac{1}{(1+\tau)^{k-1}} \frac{\Gamma\left(k - \frac{d}{2p}\right)}{\Gamma(k)} \le\; \Gamma\left( \frac{2p-d}{2p} \right) \left( 1 + \left[ \frac{\tau}{2} \right]^{(d-2p)/(2p)} \right).
\end{align*}
Combining~\eqref{zLs26} and the above inequality gives
\begin{align*}
	\|z_n\|_\infty \le\; & \frac{\|G_1\|_1}{(1+\tau)^n} \|z_0\|_\infty + \|G_1\|_{p'} F_p \left( \frac{2}{1+\tau} \right)^{(2p-d)/(2p)} \Gamma\left( \frac{2p-d}{2p} \right) \left( 1 + \tau^{(2p-d)/(2p)} \right) \\
	& \le  \|G_1\|_1 \|z_0\|_\infty + 4 \|G_1\|_{p'} F_p \Gamma\left( \frac{2p-d}{2p} \right).
	\eqntag
	\label{zLs27} 
\end{align*}

Similarly, in view of~\eqref{zLs25}, it follows from~\eqref{zLs23}, \eqref{zLs24b} and Young's inequality (for convolutions) that
\begin{align*}
	\|\nabla z_n\|_\infty \le\; & \|\mathcal{B}_\tau^n\|_1 \|\nabla z_0\|_\infty + \tau \sum_{k=1}^{n} \|\nabla \mathcal{B}_\tau^{k}\|_{p'} \| f_{n+1-k}\|_p \\
	\le\; & \frac{\|G_1\|_1}{(1+\tau)^n} \|\nabla z_0\|_\infty \\
	& + \|\nabla G_1\|_{p'} F_p \left( \frac{\tau}{1+\tau} \right)^{(p-d)/(2p)} \sum_{k=1}^n \frac{1}{(1+\tau)^{k-1}} \frac{\Gamma\left(k - \frac{1}{2} - \frac{d}{2p}\right)}{\Gamma(k)}.
\end{align*}
Arguing as above, we find that
\begin{align*}
	\sum_{k=1}^n \frac{1}{(1+\tau)^{k-1}} \frac{\Gamma\left(k - \frac{1}{2} - \frac{d}{2p}\right)}{\Gamma(k)} \le \Gamma\left( \frac{1}{2} - \frac{d}{2p} \right) \left[ 1 + \left( \frac{\tau}{2} \right)^{(d-p)/2p} \right].
\end{align*}
Combining the above two inequalities, we end up with
\begin{align*}
	\|\nabla z_n\|_\infty \le\; & \frac{\|G_1\|_1}{(1+\tau)^n} \|\nabla z_0\|_\infty + \|\nabla G_1\|_{p'} F_p \left( \frac{2}{1+\tau} \right)^{(p-d)/(2p)} \Gamma\left( \frac{1}{2} - \frac{d}{2p} \right) \left[ 1 + \tau^{(p-d)/2p} \right] \\
	\le\; & \|G_1\|_1 \|\nabla z_0\|_\infty + 4 \|\nabla G_1\|_{p'} F_p \Gamma\left( \frac{1}{2} - \frac{d}{2p} \right).
\end{align*}
Recalling~\eqref{zLs27}, we have shown~\eqref{zLs17a}.

We finally turn to the proof of~\eqref{zLs17b} and set $L_n := z_n - \Delta z_n$ for $n\in\mathbb{N}$. By~\eqref{zLs00}, $L_n$ solves
\begin{equation*}
	\frac{L_n-L_{n-1}}{\tau} - \Delta L_n + L_n = f_n - \Delta f_n \;\;\text{ in }\;\; \mathbb{R}^d, \qquad n\ge 1,
\end{equation*}
which is nothing but~\eqref{zLs00} with right-hand side $f_n-\Delta f_n$ instead of $f_n$. We then infer from~\eqref{zLs23} that
\begin{align*}
	L_n =\; & \mathcal{B}_\tau^n * L_0 + \tau \sum_{k=1}^n \mathcal{B}_\tau^{n+1-k} * (f_k-\Delta f_k) \\
	=\; & \mathcal{B}_\tau^n * L_0 + \tau \sum_{k=1}^n \mathcal{B}_\tau^{n+1-k} * f_k  - \tau \sum_{k=1}^n \nabla \mathcal{B}_\tau^{n+1-k} * \nabla f_k,
\end{align*}
whence, using~\eqref{zLs16b}, \eqref{zLs24} and Young's inequality (for convolutions),
\begin{align*}
	\|L_n\|_\infty \le\; & \|\mathcal{B}_\tau^n\|_1 \|L_0\|_\infty + \tau \sum_{k=1}^n \|\mathcal{B}_\tau^{k}\|_1 \|f_{n+1-k}\|_\infty  + \tau \sum_{k=1}^n \|\nabla \mathcal{B}_\tau^{k}\|_1 \|\nabla f_{n+1-k}\|_\infty \\
	& \le \frac{\|G_1\|_1}{(1+\tau)^n} \|z_0-\Delta z_0\|_\infty + \frac{\tau}{1+\tau} \|G_1\|_1 F_{1,\infty} \sum_{k=1}^n \frac{1}{(1+\tau)^{k-1}} \\
	& + \left( \frac{\tau}{1+\tau} \right)^{1/2} \|\nabla G_1\|_1 F_{1,\infty} \sum_{k=1}^n \frac{1}{(1+\tau)^{k-1}} \frac{\Gamma\left( k-\frac{1}{2} \right)}{\Gamma(k)}.
\end{align*}
Since
\begin{equation*}
	\sum_{k=1}^n \frac{1}{(1+\tau)^{k-1}} = \frac{1+\tau}{\tau} \left( 1 - \frac{1}{(1+\tau)^n} \right) \le \frac{1+\tau}{\tau}
\end{equation*}
and
\begin{align*}
	\sum_{k=1}^n \frac{1}{(1+\tau)^{k-1}} \frac{\Gamma\left( k-\frac{1}{2} \right)}{\Gamma(k)} \le\; & \Gamma\left( \frac{1}{2} \right) + \sum_{k=2}^{n} \frac{1}{(1+\tau)^{k-1}} \frac{1}{\sqrt{k-1}} \\
	& \le \Gamma\left( \frac{1}{2} \right) \left( 1 + \frac{1}{\sqrt{\ln{(1+\tau)}}} \right) \\
	& \le \Gamma\left( \frac{1}{2} \right) \left( 1 + \sqrt{\frac{2}{\tau}} \right)
\end{align*}
by Gautschi's inequality, we conclude that
\begin{align*}
	\|L_n\|_\infty \le\; & \frac{\|G_1\|_1}{(1+\tau)^n} \|z_0-\Delta z_0\|_\infty + \|G_1\|_1 F_{1,\infty} \\
	& + \left( \frac{2}{1+\tau} \right)^{1/2} \|\nabla G_1\|_1 F_{1,\infty}  \Gamma\left( \frac{1}{2} \right) \left( 1 + \sqrt{\tau} \right) \\
	\le\; & \|G_1\|_1 \|z_0-\Delta z_0\|_\infty + \|G_1\|_1 F_{1,\infty} + 4 \|\nabla G_1\|_1 F_{1,\infty}  \Gamma\left( \frac{1}{2} \right) 
\end{align*}
and thereby establish~\eqref{zLs17b}.
\end{proof}

\typeout{} %
\bibliographystyle{siam}
\bibliography{HL_DS}

\begin{thebibliography}{10}

\bibitem{Ali1979}
{\sc N.~D. Alikakos}, {\em {$L\sp{p}$}\ bounds of solutions of
  reaction-diffusion equations}, Comm. Partial Differential Equations, 4
  (1979), pp.~827--868.

\bibitem{AGS2005}
{\sc L.~Ambrosio, N.~Gigli, and G.~Savar{\'e}}, {\em Gradient flows in metric
  spaces and in the space of probability measures}, Basel: Birkh{\"a}user,
  2005.

\bibitem{APW2002}
{\sc A.~Ashyralyev, S.~Piskarev, and L.~Weis}, {\em On well-posedness of
  difference schemes for abstract parabolic equations in {$L^p([0,T];E)$}
  spaces}, Numer. Funct. Anal. Optim., 23 (2002), pp.~669--693.

\bibitem{BCKKLL2015}
{\sc A.~Blanchet, J.~A. Carrillo, D.~Kinderlehrer, M.~Kowalczyk, {\relax
  Ph}.~Lauren\c{c}ot, and S.~Lisini}, {\em A hybrid variational principle for
  the {K}eller-{S}egel system in {$\Bbb R^2$}}, ESAIM Math. Model. Numer.
  Anal., 49 (2015), pp.~1553--1576.

\bibitem{BlLa2013}
{\sc A.~Blanchet and {\relax Ph}.~Lauren{\c{c}}ot}, {\em The
  parabolic-parabolic {K}eller-{S}egel system with critical diffusion as a
  gradient flow in {$\Bbb R^d,\ d\ge3$}}, Comm. Partial Differential Equations,
  38 (2013), pp.~658--686.

\bibitem{CLW2012}
{\sc L.~Chen, J.-G. Liu, and J.~Wang}, {\em Multidimensional degenerate
  {K}eller-{S}egel system with critical diffusion exponent {$2n/(n+2)$}}, SIAM
  J. Math. Anal., 44 (2012), pp.~1077--1102.

\bibitem{ChWa2014}
{\sc L.~Chen and J.~Wang}, {\em Exact criterion for global existence and blow
  up to a degenerate {K}eller-{S}egel system}, Doc. Math., 19 (2014),
  pp.~103--120.

\bibitem{DiWa2019}
{\sc M.~Ding and W.~Wang}, {\em Global boundedness in a quasilinear fully
  parabolic chemotaxis system with indirect signal production}, Discrete
  Contin. Dyn. Syst. Ser. B, 24 (2019), pp.~4665--4684.

\bibitem{EnNa2000}
{\sc K.-J. Engel and R.~Nagel}, {\em One-parameter semigroups for linear
  evolution equations}, vol.~194 of Graduate Texts in Mathematics,
  Springer-Verlag, New York, 2000.
\newblock With contributions by S. Brendle, M. Campiti, T. Hahn, G. Metafune,
  G. Nickel, D. Pallara, C. Perazzoli, A. Rhandi, S. Romanelli and R.
  Schnaubelt.

\bibitem{FuSe2017}
{\sc K.~Fujie and T.~Senba}, {\em Application of an {A}dams type inequality to
  a two-chemical substances chemotaxis system}, J. Differential Equations, 263
  (2017), pp.~88--148.

\bibitem{FuSe2019}
\leavevmode\vrule height 2pt depth -1.6pt width 23pt, {\em Blowup of solutions
  to a two-chemical substances chemotaxis system in the critical dimension}, J.
  Differential Equations, 266 (2019), pp.~942--976.

\bibitem{HTZ2025}
{\sc M.~Herda, A.~Trescases, and A.~Zurek}, {\em A finite volume scheme for the
  local sensing chemotaxis model}, SMAI J. Comput. Math., 11 (2025),
  pp.~637--676.

\bibitem{Hos2026}
{\sc T.~Hosono}, {\em On the {C}auchy problem of a chemotaxis system with
  indirect signal production}, in Mathematical {A}nalysis and {A}pproximation
  of {PDE}-{C}hemotaxis {M}odels, vol.~43 of SEMA SIMAI Springer Ser.,
  Springer, Cham, 2026, pp.~61--82.

\bibitem{HoLa2025}
{\sc T.~Hosono and {\relax Ph}.~Lauren\c{c}ot}, {\em Global existence and
  boundedness of solutions to a fully parabolic chemotaxis system with indirect
  signal production in {$\Bbb R^4$}}, J. Differential Equations, 416 (2025),
  pp.~2085--2133.

\bibitem{IsYo2020}
{\sc S.~Ishida and T.~Yokota}, {\em Boundedness in a quasilinear fully
  parabolic {K}eller-{S}egel system via maximal {S}obolev regularity}, Discrete
  Contin. Dyn. Syst. Ser. S, 13 (2020), pp.~212--232.

\bibitem{JKO1998}
{\sc R.~Jordan, D.~Kinderlehrer, and F.~Otto}, {\em The variational formulation
  of the {F}okker-{P}lanck equation}, SIAM J. Math. Anal., 29 (1998),
  pp.~1--17.

\bibitem{KeSe1971}
{\sc E.~F. Keller and L.~A. Segel}, {\em Model for chemotaxis}, J. Theoret.
  Biol., 30 (1971), pp.~225--234.

\bibitem{KNO2015}
{\sc A.~Kimijima, K.~Nakagawa, and T.~Ogawa}, {\em Threshold of global behavior
  of solutions to a degenerate drift-diffusion system in between two critical
  exponents}, Calc. Var. Partial Differential Equations, 53 (2015),
  pp.~441--472.

\bibitem{KLL2016}
{\sc B.~Kov{\'a}cs, B.~Li, and C.~Lubich}, {\em A-stable time discretizations
  preserve maximal parabolic regularity}, SIAM J. Numer. Anal., 54 (2016),
  pp.~3600--3624.

\bibitem{Lau1994}
{\sc {\relax Ph}.~Lauren\c{c}ot}, {\em Solutions to a {P}enrose-{F}ife model of
  phase-field type}, J. Math. Anal. Appl., 185 (1994), pp.~262--274.

\bibitem{LiWa2026}
{\sc K.~Lin and S.~Wang}, {\em Sharp critical mass criterion for the fully
  parabolic {K}eller-{S}egel system with the intermediate exponent},
  Nonlinearity, 39 (2026), pp.~Paper No. 015008, 25.

\bibitem{MLL2025}
{\sc X.~Mao, M.~Liu, and Y.~Li}, {\em Finite-time blowup in a fully parabolic
  chemotaxis model involving indirect signal production}.
\newblock Preprint, {arXiv}:2503.12439 [math.{AP}] (2025), 2025.

\bibitem{MMS2009}
{\sc D.~Matthes, R.~J. McCann, and G.~Savar{\'e}}, {\em A family of nonlinear
  fourth order equations of gradient flow type}, Commun. Partial Differ.
  Equations, 34 (2009), pp.~1352--1397.

\bibitem{Mim2017}
{\sc Y.~Mimura}, {\em The variational formulation of the fully parabolic
  {K}eller-{S}egel system with degenerate diffusion}, J. Differential
  Equations, 263 (2017), pp.~1477--1521.

\bibitem{Mim2024b}
\leavevmode\vrule height 2pt depth -1.6pt width 23pt, {\em Formulation of
  {C}himera gradient flows for chemotaxis systems with indirect signal
  production and degenerate diffusion}.
\newblock Preprint, {arXiv}:2406.14536 [math.{AP}] (2024), 2024.

\bibitem{Mim2024a}
\leavevmode\vrule height 2pt depth -1.6pt width 23pt, {\em Global existence of
  solutions to parabolic-parabolic {Keller}-{Segel} system in between two
  critical exponents}, Adv. Math. Sci. Appl., 33 (2024), pp.~77--96.

\bibitem{Oga2011}
{\sc T.~Ogawa}, {\em The degenerate drift-diffusion system with the {S}obolev
  critical exponent}, Discrete Contin. Dyn. Syst. Ser. S, 4 (2011),
  pp.~875--886.

\bibitem{Sug2006}
{\sc Y.~Sugiyama}, {\em Global existence in sub-critical cases and finite time
  blow-up in super-critical cases to degenerate {Keller}-{Segel} systems.},
  Differ. Integral Equ., 19 (2006), pp.~841--876.

\end{thebibliography}

\end{document}